\documentclass[11pt,reqno]{amsart}
\usepackage[T1]{fontenc}
\usepackage{lmodern}
\usepackage{amsmath,amssymb,amsthm,mathtools,esint}
\usepackage[margin=1in]{geometry}
\usepackage{microtype}
\microtypesetup{expansion=false}
\usepackage{enumitem,booktabs,array,needspace}
\usepackage{xcolor}
\usepackage[colorlinks=true,linkcolor=blue!40!black,
 citecolor=blue!40!black,urlcolor=blue!40!black]{hyperref}
\hypersetup{pdftitle={Boundary Vanishing Orders and Asymptotic Expansions for Planar Harmonic Functions with a Sublinear Robin Condition}}
\newtheorem{theorem}{Theorem}[section]
\newtheorem{lemma}[theorem]{Lemma}
\newtheorem{proposition}[theorem]{Proposition}
\newtheorem{corollary}[theorem]{Corollary}
\theoremstyle{definition}

\theoremstyle{remark}
\newtheorem{remark}{Remark}[section]
\numberwithin{equation}{section}
\newcommand{\R}{\mathbb R}
\newcommand{\C}{\mathbb C}
\newcommand{\N}{\mathbb N}
\newcommand{\Om}{\Omega}
\newcommand{\HH}{\R^2_+}
\newcommand{\ord}{\operatorname{ord}}
\newcommand{\Lip}{\operatorname{Lip}}

\newcommand{\Span}{\operatorname{span}}
\newcommand{\ind}{\mathbf 1}
\newcommand{\sgn}{\operatorname{sgn}}
\newcommand{\Rea}{\operatorname{Re}}
\newcommand{\Ima}{\operatorname{Im}}
\newcommand{\eps}{\varepsilon}
\newcommand{\dd}{\,d}

\DeclareMathOperator{\Log}{Log}
\setlist[enumerate]{label=\textup{(\arabic*)},leftmargin=2.2em}
\allowdisplaybreaks[1]

\title[Boundary Vanishing Orders]
{Boundary Vanishing Orders and Asymptotic Expansions\\
for Planar Harmonic Functions\\
with a Sublinear Robin Condition}
\author[H. Liu]{Hairong Liu\textsuperscript{1}}
\address{\textsuperscript{1}School of Mathematics and Statistics, Nanjing University of Science and Technology, Nanjing 210094, P.R. China}
\email{hrliu@njust.edu.cn}
\author[L. Tian]{Long Tian\textsuperscript{2}}
\address{\textsuperscript{2}School of Mathematics and Statistics, Nanjing University of Science and Technology, Nanjing 210094, P.R. China}
\email{tianlong19850812@163.com}
\author[X. Yang]{Xiaoping Yang\textsuperscript{3}}
\address{\textsuperscript{3}School of Mathematics, Nanjing University, Nanjing 210093, P.R. China}
\email{xpyang@nju.edu.cn}
\hypersetup{pdfauthor={Hairong Liu, Long Tian, Xiaoping Yang}}
\date{}
\subjclass[2020]{35J25, 35B40, 35B60}

\keywords{Sublinear Robin condition, boundary vanishing order, nodal domains,
 conformal mapping, Dini mean oscillation}

\begin{document}
\begin{abstract}
We study boundary vanishing orders and asymptotic expansions of planar
harmonic functions satisfying the sublinear Robin condition
\[
 \partial_\nu u=\lambda u-|u|^{p-1}u,
 \qquad \lambda\ge0,\quad 0<p<1.
\]
For nontrivial global solutions on bounded connected Lipschitz domains,
we bound the number of nodal domains by the number of Steklov eigenvalues
strictly below $\lambda$ and establish finite
vanishing order at every boundary point.
The boundary scaling gives the critical exponent $\kappa=1/(1-p)$.
On $C^1$ boundaries, every positive subcritical vanishing order is an
integer, and the leading term is a unique homogeneous Neumann profile in
conformal coordinates. At a noninteger critical order, the rescaled
solution converges in normalized graph coordinates to an explicit
homogeneous profile determined by $p$ up to sign. Every supercritical
bulk vanishing order $m$ is an integer, with a unique homogeneous
Dirichlet profile in conformal coordinates. The boundary trace has
vanishing order $(m-1)/p>m$. Integer critical vanishing orders are
excluded if the boundary is $C^1$ with Dini mean oscillation.
\end{abstract}

\maketitle

\section{Introduction}\label{paper:sec:introduction}

Let $\Om\subset\R^2$ be a bounded connected domain. We consider
the restoring sublinear boundary problem
\begin{equation}\label{eq:problem}
\begin{cases}
  \Delta u=0 &\text{in }\Om,\\
  \partial_\nu u=\lambda u-|u|^{p-1}u &\text{on }\partial\Om,
 \end{cases}
 \end{equation}
where $\lambda\geq0$, $0<p<1$, and $\nu$ is the outward unit
normal. Our aim is to
understand how this boundary nonlinearity and the geometry of
$\partial\Om$ determine the boundary vanishing orders of nontrivial solutions.

\subsection{Background and related work}
\label{paper:subsec:background}

Unique continuation for elliptic equations asks whether vanishing on
an open set or to infinite order at a point forces a solution to
vanish identically. Quantitative results control doubling, vanishing
orders, and the size of nodal and singular sets. Classical tools
include Carleman estimates~\cite{JerisonKenig} and frequency-function
methods~\cite{GarofaloLin,LinNodal}. Boundary unique continuation
also depends on the boundary condition and domain regularity.

For harmonic functions vanishing continuously on an open boundary patch,
Adolfsson, Escauriaza and Kenig~\cite{AEK} proved uniqueness when the normal
derivative vanishes on a set of positive surface measure on a convex
domain. Adolfsson and Escauriaza~\cite{AE} developed boundary doubling
and unique continuation for Dirichlet conditions on Dini domains and
homogeneous Neumann conditions on $C^{1,1}$ domains. For the Dirichlet
problem, Kukavica and Nystr\"om~\cite{KukavicaNystrom} established
positive-measure uniqueness on Dini domains, and Tolsa~\cite{Tolsa}
extended it to $C^1$ domains and Lipschitz domains with sufficiently
small Lipschitz constant.

Banerjee and Garofalo~\cite{BanerjeeGarofalo} obtained quantitative
boundary vanishing-order bounds for variable-coefficient Schr\"odinger
equations with zero Dirichlet data on $C^{1,\mathrm{Dini}}$ domains.
Under the same geometric regularity, Kenig and
Zhao~\cite{KenigZhaoSingular} estimated the singular set up to a
Dirichlet boundary. They also established a unique leading homogeneous
harmonic polynomial at each point of a zero Dirichlet patch, with an
error controlled by the Dini geometry~\cite{KenigZhaoExpansion}.
Kriventsov and Li~\cite{KriventsovLi} related asymptotic expansions at
boundary points with tangent cones to quantitative geometric
assumptions. Zhu and Zhuge~\cite{ZhuZhuge} obtained nodal-set estimates
for Dirichlet eigenfunctions on quasiconvex Lipschitz domains.

In two dimensions, Vita~\cite{Vita} used conformal mapping to prove
boundary Cauchy uniqueness on chord-arc domains and local finiteness
of critical points near a zero Dirichlet patch on $C^{1,\mathrm{DMO}}$
boundaries. Here the tangent has Dini mean oscillation, a condition
weaker than pointwise Dini continuity that yields nondegenerate
conformal maps near the boundary. For the nonlinear Robin problem,
the trace need not vanish on an entire boundary patch, so odd
reflection and Dirichlet results do not directly yield the desired
asymptotics.

For linear Robin problems, Li and Wang~\cite{LiWang} proved quantitative
vanishing-order estimates and boundary unique continuation on
$C^{1,1}$ domains with regular Robin potentials. Li~\cite{Li} extended
unique continuation on such domains to nonsmooth potentials satisfying
suitable integrability conditions, including $L^{d-1+\varepsilon}$
in dimension $d$. The formal Robin coefficient $\lambda-|u|^{p-1}$
in \eqref{eq:problem} is singular at boundary zeros, and its
integrability cannot be assumed before the vanishing behavior is
understood. A direct linear reduction also loses the sign and
nonlinear structure essential here.

For sublinear interior equations such as $-\Delta v=|v|^{p-1}v$,
$0<p<1$, the quotient of the nonlinearity by $v$ is likewise singular
near the nodal set. Parini and Weth~\cite{PariniWeth} proved that least
energy nodal solutions of a sublinear Neumann problem have a zero set
of Lebesgue measure zero. Soave and Weth~\cite{SoaveWeth} established
weak unique continuation for a general class of sublinear elliptic
equations without a minimizing assumption. R\"uland~\cite{Ruland}
proved weak and strong unique continuation, and under additional
hypotheses continuation from sets of positive measure, using Carleman
estimates that incorporate the sublinear term. In \eqref{eq:problem},
by contrast, the reaction acts only on the boundary and the interior
equation remains harmonic.

Soave and Terracini~\cite{SoaveTerracini} classified vanishing orders
and studied nodal sets for two-phase sublinear interior equations.
For the model with power $p$, the admissible orders are integers below
the critical exponent $2/(1-p)$ and the critical exponent itself.
Their analysis uses nondegeneracy, blow-up arguments, and monotonicity
formulas. When the two phases have different powers, Soave and
Tortone~\cite{SoaveTortone} classified vanishing orders for local
minimizers.

Fractional equations provide a closer scaling analogue. The extension
of Caffarelli and Silvestre~\cite{CaffarelliSilvestre} realizes
fractional Laplacians as weighted Dirichlet-to-Neumann operators;
for $s=1/2$, it gives the ordinary Dirichlet-to-Neumann map for harmonic
functions on a half-space. For nonlocal sublinear equations,
Tortone~\cite{Tortone} proved strong unique continuation, classified
vanishing orders, and studied nodal stratification for weak stationary
solutions. The critical exponent $2s/(1-p)$ becomes $1/(1-p)$ at
$s=1/2$. However, the fractional equation there has a positive
sublinear reaction, whereas \eqref{eq:problem} has a negative,
restoring reaction in the outward-flux convention. Thus its upper
bound on vanishing orders does not transfer to our problem. We show
that solutions with supercritical bulk order have boundary traces
vanishing at a strictly higher order.

\subsection{Critical scaling and main ideas}
\label{paper:subsec:scaling}

The scaling behind this phenomenon can be seen directly. If a
solution has leading size $r^k$, then its normal derivative has
scale $r^{k-1}$, the sublinear term has scale $r^{pk}$, and the
linear term has scale $r^k$. Hence the distinguished boundary
exponent is
\[
 \kappa=\frac1{1-p}.
\]
For $k<\kappa$, the rescaled reaction disappears and the limiting
boundary condition is homogeneous Neumann. At $k=\kappa$, the
normal derivative and the sublinear term balance, so the nonlinear
law remains in the limiting problem. For $k>\kappa$, a nonzero
trace of the leading bulk profile would be incompatible with the
balance of powers. The leading harmonic term must therefore have
zero trace, while a finer boundary term supplies the nonlinear
flux. If the bulk degree is $m$, this balance predicts the trace
degree
\[
 p\ell=m-1,\qquad \ell=\frac{m-1}{p}>m.
\]
Thus the bulk and boundary vanishing orders must be distinguished,
even though they arise from the same harmonic function.

Boundary regularity determines how the conformal asymptotics pass
to the original coordinates. Under $C^{1,\mathrm{DMO}}$ regularity,
the conformal map has a nonzero boundary derivative, which gives
nondegenerate physical expansions in the subcritical and
supercritical regimes. We also establish finite positive limits
for the corresponding normalized bulk densities. At noninteger
critical order, physical graph rescaling gives the profile and
density directly under $C^1$ regularity.

The proofs combine spectral, planar, and nonlinear arguments.
The variational characterization of the Steklov spectrum,
reviewed in~\cite{GP}, gives a strict Rayleigh-quotient comparison
on nodal restrictions. Planar separation and the weighted boundary
sign theorem of Baouendi and Rothschild~\cite{BR} then connect
nodal information with boundary unique continuation. For the
local classification we use conformal coordinates, nonlinear
boundary comparison, and critical blow-up arguments based on
Weiss-type monotonicity; the latter belongs to the homogeneity
framework developed in~\cite{Weiss} and adapted to sublinear
problems in~\cite{SoaveTerracini,Tortone}.

The geometric estimates draw on quasiconformal distortion and
boundary behavior of conformal maps~\cite{Mori,Pom}, the Lindel\"of
tangent-angle relation~\cite{Kalaj},
and the nondegenerate conformal maps of~\cite{Vita}.
Conormal regularity with Dini mean oscillation coefficients
\cite{DongLeeKim} supplies the boundary differentiability needed
for the physical normal-flux asymptotics.

\subsection{Main results}
\label{paper:subsec:main-results}
\label{paper:subsec:setting}
We consider real-valued solutions of \eqref{eq:problem} in the class
\begin{equation}\label{eq:class}
 u\in H^1(\Om)\cap C(\overline\Om),\qquad u\not\equiv0.
\end{equation}
Harmonicity makes every solution smooth in the interior. The Robin
condition is understood in terms of normal traces and holds
pointwise wherever the solution is $C^1$ up to the boundary.
Under $C^{1,\mathrm{DMO}}$ regularity, this boundary differentiability
is established in Lemma~\ref{t15:lem:boundary-c1}.
For $q\in\partial\Om$, choose $R_0>0$ so that
$\partial\Om\cap B_{R_0}(q)$ is a Lipschitz graph. The \emph{local problem} is
\begin{equation}\label{paper:eq:local-problem}
 \Delta u=0\quad\text{in }\Om\cap B_{R_0}(q),\qquad
 \partial_\nu u=\lambda u-|u|^{p-1}u
 \quad\text{on }\Gamma:=\partial\Om\cap B_{R_0}(q).
\end{equation}
For the local problem, we require
\begin{equation}\label{paper:eq:local-class}
 u\in H^1(\Om\cap B_R(q))
       \cap C(\overline\Om\cap B_R(q))
 \quad\text{for every }0<R<R_0.
\end{equation}

We define the \emph{bulk $L^2$-vanishing
order}, or simply the \emph{vanishing order}, of $u$ at $q$ by
\begin{equation}\label{eq:order}
 \ord_q u:=\sup\left\{\alpha\ge0:
  \int_{\Om\cap B_r(q)}u^2\dd X=O(r^{2\alpha+2})
  \text{ as }r\rightarrow0\right\}.
\end{equation}
The value $+\infty$ is initially allowed, and the supremum need not
be attained. We also define the \emph{boundary $L^\infty$-vanishing
order} and the \emph{boundary $L^2$-vanishing order}, respectively, by
\begin{equation}\label{paper:eq:boundary-orders}
 \begin{split}
 \ord_q^{\partial,\infty} u
 &=\sup\left\{\beta\ge0:
    \sup_{\partial\Om\cap B_r(q)}|u|=O(r^\beta)\right\},\\
 \ord_q^\partial u
 &=\sup\left\{\beta\ge0:
    \int_{\partial\Om\cap B_r(q)}u^2\dd S
                   =O(r^{2\beta+1})\right\}.
 \end{split}
\end{equation}

The boundary regularity class $C^{1,\mathrm{DMO}}$,
the local conformal maps, and the coordinate
conventions are defined in Section~\ref{sec:prelim}.

\subsubsection{Nodal domains and boundary unique continuation}
\label{paper:subsubsec:nodal-sucp}

A nodal domain is a connected component of $\{u\ne0\}\cap\Om$,
and $d(u)$ denotes their number. Write the Steklov spectrum, with
multiplicity, as
\[
 0=\sigma_0(\Om)<\sigma_1(\Om)\le\sigma_2(\Om)\le\cdots,
 \qquad N_S^{<}(t)=\#\{j\ge0:\sigma_j(\Om)<t\}.
\]

\begin{theorem}\label{thm:main}
Let $\Om\subset\R^2$ be a bounded connected Lipschitz domain, let
$0<p<1$ and $\lambda>0$, and let $u\in H^{1}(\Om)\cap C(\overline{\Om})$ be a nontrivial solution of 
\eqref{eq:problem}. Then
\begin{equation}\label{eq:nodal-main}
 d(u)\le N_S^{<}(\lambda)<\infty,
\end{equation}
and
\begin{equation}\label{eq:finite-main}
 \ord_q u<\infty\qquad\text{for every }q\in\partial\Om.
\end{equation}
\end{theorem}

\begin{remark}
If $\lambda=0$, Green's formula for \eqref{eq:problem}, with $u$
as the test function, gives
\[
 \int_\Om|\nabla u|^2\dd X+
 \int_{\partial\Om}|u|^{p+1}\dd S=0.
\]
Hence $u\equiv0$.
\end{remark}

The finite vanishing order conclusion of the above theorem also
has a local analogue, namely,
\begin{proposition}[Local strong unique continuation on a planar Lipschitz boundary]
\label{t15:local-sucp}
Let $q\in\partial\Omega$ be a Lipschitz boundary point of a planar
domain $\Omega$, and let $u$ be a real local solution satisfying
\eqref{paper:eq:local-problem}--\eqref{paper:eq:local-class}, with
$0<p<1$ and $\lambda\ge0$. If $u(q)=0$ and
\[
 \int_{\Omega\cap B_r(q)}u^2=O(r^{2N+2})
 \quad\hbox{for every }N\geq1,
\]
then $u$ vanishes identically in a neighborhood of $q$.
Consequently, if $u$ is not identically zero near $q$, then its
vanishing order at $q$ is finite.
\end{proposition}

\subsubsection{Subcritical vanishing order and Neumann profiles}
\label{paper:subsubsec:subcritical}

\begin{theorem}\label{t12:thm:main}
Let $\Om\subset\R^2$ be a bounded connected domain whose boundary is
$C^1$ near $q\in\partial\Om$, and let $u$ be a real local solution
satisfying \eqref{paper:eq:local-problem}--\eqref{paper:eq:local-class}, with
$u(q)=0$. Suppose
\begin{equation*}
 0<k:=\ord_q u<\kappa,
\end{equation*}
and put $\delta_*:=1-(1-p)k>0$. Fix any local conformal
map $\Phi$ at $q$ and put $U=u\circ\Phi$.
\begin{enumerate}[label=\textup{(\roman*)}]
\item Then
the order $k$ is a positive integer. Moreover, there is a unique
$A\in\R\setminus\{0\}$, relative to the fixed conformal map, such that,
for every $0<\delta<\delta_*$,
\[
 U(z)=A\operatorname{Re}(z^k)+O(|z|^{k+\delta})
 \qquad(z\to0,\ y\ge0),
\]
uniformly on the closed half-disk.
The full family of rescalings converges as $r\rightarrow 0$:
\begin{equation}\label{eq:full-family-convergence}
    r^{-k}U(r\,\cdot)
    \longrightarrow A\,\operatorname{Re}(z^k)
\end{equation}
locally uniformly in $\mathbb{R}^{2}_{+}$ and strongly in
$H^{1}(B^{+}_{L})$ for every fixed $L>0$.
Moreover,
\begin{equation}\label{t12:eq:trace}
 \int_{-r}^rU(t,0)^2\dd t
 =\frac{2A^2}{2k+1}r^{2k+1}+o(r^{2k+1}).
\end{equation}

\item If, in addition, $\partial\Om$ is $C^{1,\mathrm{DMO}}$ near $q$,
then the pullback $U$ admits the refined asymptotic expansion
\begin{equation*}
 U(z)=A\operatorname{Re}(z^k)+O(|z|^{k+\delta_*})
 \qquad(z\to0,\ z\in\overline{B_R^+}),
\end{equation*}
with a uniform remainder.

In the original coordinates, $u$ admits the asymptotic expansion
\begin{equation}\label{t12:eq:physical-expansion}
 u(q+w)=P_k(w)+o(|w|^k)
 \qquad(w\to0,\ q+w\in\overline\Om),
\end{equation}
uniformly in the direction of approach, where
\begin{equation*}
 P_k(w):=A\operatorname{Re}\bigl((w/a_\Phi)^k\bigr)
\end{equation*}
is a nonzero homogeneous harmonic polynomial of degree $k$
satisfying the homogeneous Neumann condition on
$\partial H_{a_\Phi}$, with 
 $H_{a_\Phi}:=\{w\in\C:\operatorname{Im}(w/a_\Phi)>0\}$.

Furthermore, the normalized $L^2$-mass satisfies
\begin{equation*}
 \lim_{r\rightarrow0}r^{-2k-2}
 \int_{\Om\cap B_r(q)}u^2\dd X
 =\frac{\pi A^2}{4(k+1)|a_\Phi|^{2k}}>0.
\end{equation*}
\end{enumerate}
\end{theorem}

\Needspace{6\baselineskip}
\subsubsection{Critical vanishing and the role of boundary regularity}
\label{paper:subsubsec:critical}

The next theorem identifies the unique asymptotic profile
when the vanishing order equals the noninteger critical value
$\kappa$.

\begin{theorem}\label{t134:ni:thm:main}
Let $\Om\subset\R^2$ be a bounded connected $C^1$ domain,
let $q\in\partial\Om$, and let $u$ satisfy
\eqref{paper:eq:local-problem}--\eqref{paper:eq:local-class}.
In the normalized graph coordinates \eqref{t134:eq:graph},
assume that $q=0$, $u(0)=0$, and $\ord_0u=\kappa\notin\N$.
Set
\[
 \beta:=\frac{\pi}{2}\bigl(\kappa-\lfloor\kappa\rfloor\bigr),
 \qquad
 a_\kappa:=(\kappa\tan\beta)^{-\kappa},
\]
and define, in polar coordinates,
\[
 P_\pm(s,\theta)
 :=\pm\frac{a_\kappa}{\cos\beta}s^\kappa
       \cos(\kappa\theta-\beta),
 \qquad -\varepsilon<\theta<\pi+\varepsilon,
\]
where $\varepsilon>0$ is fixed and sufficiently small.
Then there is a unique $\sigma\in\{+,-\}$ such that,
for every $L>0$, as $r\rightarrow0$,
\begin{equation}\label{t134:ni:eq:convergence}
 r^{-\kappa}u\bigl(rx,ry+\psi(rx)\bigr)
 \longrightarrow P_\sigma(x,y)
 \quad\text{in }C^0(\overline{B_L^+})
 \text{ and strongly in }H^1(B_L^+).
\end{equation}
Moreover,
\begin{equation}\label{t134:ni:eq:physical}
 u(X)=P_\sigma(X)+o(|X|^\kappa)
 \qquad(X\to0,\ X\in\overline\Om),
\end{equation}
with a uniform remainder, and
\begin{equation}\label{t134:ni:eq:density-limit}
 \lim_{r\rightarrow0}r^{-2\kappa-2}
 \int_{\Om\cap B_r}u^2\dd X
 =D_*:=\frac{a_\kappa^2}{4(\kappa+1)\cos^2\beta}
 \left(\pi+\frac{\sin(2\beta)}{\kappa}\right)>0.
\end{equation}
\end{theorem}

\Needspace{6\baselineskip}
The next theorem shows that, if $\partial\Om$ is
$C^{1,\mathrm{DMO}}$ near $q$ and  $\kappa$
is an integer, then the vanishing order of $u$ at $q$
cannot equal $\kappa$.

\begin{theorem}\label{thm:integer-critical-order}
Let $\Om\subset\R^2$ be a bounded connected domain, and let
$q\in\partial\Om$. Assume that $\partial\Om$ is
$C^{1,\mathrm{DMO}}$ in a neighborhood of $q$ and that $u$
satisfies the local problem
\eqref{paper:eq:local-problem}--\eqref{paper:eq:local-class}.
Suppose that $u(q)=0$ and
\[
 u\not\equiv0 \quad\text{in }\Om\cap B_r(q)
 \qquad\text{for every sufficiently small }r>0.
\]
If the critical value $\kappa$ is an integer, then the
vanishing order of $u$ at $q$ cannot equal $\kappa$.
\end{theorem}

\begin{remark}
\label{rem:c1-integer-critical-example}
The $C^{1,\mathrm{DMO}}$ hypothesis in
Theorem~\ref{thm:integer-critical-order} cannot be replaced by
$C^1$ regularity. Here is a concrete local example with
$p=\tfrac12$, $\lambda=1$, and $\kappa=2$.
Set $r_*=e^{-4}$, $c=\pi/2$, and, in the upper half-disk, define
\[
 L(z)=\Log(i/z),\qquad
 U(z)=\Rea\!\left(\frac{z^2}{L(z)}\right),\qquad U(0)=0,
\]
where $\Log$ is the principal branch. For $0<|t|<r_*$, put
\[
 T(t)=\log\frac1{|t|},\qquad
 \rho(t)=
 \frac{\pi\bigl(T(t)^2+T(t)+c^2\bigr)}
 {T(t)^{1/2}\bigl(T(t)^2+c^2\bigr)^{3/2}
  \bigl(1-|t|\sqrt{T(t)/(T(t)^2+c^2)}\bigr)},
\]
and define the holomorphic map
\[
 \Phi(z)=\int_0^z
 \exp\!\left\{\frac1{\pi i}
       \int_{-r_*}^{r_*}\frac{\log\rho(t)}{t-\zeta}\dd t
      \right\}\dd\zeta,\qquad \Phi(0)=0,
\]
with the integral from $0$ understood in the improper sense.
For sufficiently small $R>0$, the map $\Phi$ is conformal on
$B_R^+$. The domain $\Om=\Phi(B_R^+)$ has a $C^1$ boundary near
$0$, with horizontal tangent there, and
$u=U\circ\Phi^{-1}$ is a nontrivial local solution of
\eqref{paper:eq:local-problem}--\eqref{paper:eq:local-class}
with the above $p$ and $\lambda$. Moreover,
\[
 \int_{\Om\cap B_r}u^2\dd X
 \asymp r^6\left(\log\frac1r\right)^4
 \qquad(r\to0),
\]
so that $u(0)=0$ and $\ord_0u=2=\kappa$.
Thus integer critical vanishing order can occur on a $C^1$ boundary.
\end{remark}

\subsubsection{Supercritical vanishing, boundary traces and normal fluxes}
\label{paper:subsubsec:supercritical}

\begin{theorem}\label{t15:thm:main}
Let $\Om\subset\R^2$ be a bounded connected domain whose boundary is
$C^1$ near $q\in\partial\Om$, and let $u$ be a real local solution
satisfying \eqref{paper:eq:local-problem}--\eqref{paper:eq:local-class}, with
$u(q)=0$. Suppose
\[
 m:=\ord_q u>\kappa.
\]
Fix any local conformal map $\Phi$ at $q$ and put
$U=u\circ\Phi$.
\begin{enumerate}[label=\textup{(\roman*)}]
\item Then the vanishing order $m$ is finite and $m$ is an integer.
Put
\begin{equation*}
 \ell:=\frac{m-1}{p}>m.
\end{equation*}
There is a unique $A\in\R\setminus\{0\}$, relative to the fixed
conformal map, such that, for every
$0<\delta<\min\{1,\ell-m\}$,
\begin{equation}\label{t15:eq:conformal-expansion}
 U(z)=A\Ima(z^m)+O(|z|^{m+\delta})
 \qquad(z\to0,\ y\ge0),
\end{equation}
uniformly on the closed half-disk.
The full family of rescalings converges as $r\rightarrow0$:
\begin{equation*}
 r^{-m}U(rz)\longrightarrow A\Ima(z^m)
\end{equation*}
in $C^0(\overline{B_L^+})$ and strongly in $H^1(B_L^+)$ for every
fixed $L>0$. Moreover, the boundary trace in the original
coordinates satisfies, for every
$0<\beta<\ell$,
\begin{equation*}
 |u(Q)|\le C_\beta|Q-q|^\beta
 \qquad(Q\in\partial\Om,\ Q\to q).
\end{equation*}
The two boundary vanishing orders are exactly
\begin{equation*}
 \ord_q^{\partial,\infty} u=\ord_q^\partial u=\ell.
\end{equation*}

\item If, in addition, $\partial\Om$ is $C^{1,\mathrm{DMO}}$ near $q$,
then, after reducing $R$, the same conformal map $\Phi$ belongs to
$C^1(\overline{B_R^+})$ and satisfies
\[
 a_\Phi:=\Phi'(0)\ne0,\qquad
 0<c\le|\Phi'(z)|\le C
 \quad\text{for all }z\in\overline{B_R^+}.
\]
Put
\[
 a_0:=|a_\Phi|,\qquad
 e=a_\Phi/a_0,\qquad B=A/a_0^m\ne0.
\]
Choose the associated  graph parametrization
\begin{equation*}
 \gamma(s)=q+e\{s+i\psi(s)\},\qquad
 \psi(0)=\psi'(0)=0,
\end{equation*}
with the domain on the side of the inward tangent normal $ie$.
In the original coordinates, $u$ admits the asymptotic expansion
\begin{equation*}
 u(q+w)=B\Ima\bigl((\overline e w)^m\bigr)+o(|w|^m)
 \qquad(w\to0,\ q+w\in\overline\Om),
\end{equation*}
uniformly in the direction of approach. With
\[
 F_r(x,y)=q+e\{rx+i(ry+\psi(rx))\},
\]
the full family of rescalings converges as $r\rightarrow0$:
\begin{equation*}
 r^{-m}u(F_r(x,y))\longrightarrow B\Ima((x+iy)^m)
\end{equation*}
in $C^0(\overline{B_L^+})$ and strongly in $H^1(B_L^+)$ for every
fixed $L>0$. Furthermore, the normalized $L^2$-mass satisfies
\begin{equation*}
 \lim_{r\rightarrow0}r^{-2m-2}
       \int_{\Om\cap B_r(q)}u^2\dd X
 =\frac{\pi B^2}{4(m+1)}>0.
\end{equation*}
Finally, define
\begin{equation*}
 G_B(x)=\sgn(Bx^{m-1})(m|B|)^{1/p}|x|^\ell,\qquad G_B(0)=0.
\end{equation*}
Then, as $r\rightarrow0$, the boundary trace and flux satisfy
\begin{align}
 r^{-\ell}u(\gamma(rx))&\longrightarrow G_B(x),
       \label{t15:eq:physical-fine-trace}\\
 r^{1-m}|u(\gamma(rx))|^{p-1}u(\gamma(rx))&\longrightarrow mBx^{m-1},
       \notag\\
 r^{1-m}\partial_{\nu_{\rm in}}u(\gamma(rx))
       &\longrightarrow mBx^{m-1},
       \label{t15:eq:physical-normal}
\end{align}
locally uniformly for $x\in\R$. Here $\nu_{\rm in}$ is the
inward unit normal at the displayed boundary point. The solution
is locally $C^1$ up to this boundary, so its normal derivative
in \eqref{t15:eq:physical-normal} is classical.

\end{enumerate}
\end{theorem}

\subsection{Organization of the paper}
\label{paper:subsec:organization}
Section~\ref{sec:prelim} introduces the boundary regularity assumptions
and conformal coordinates and establishes the local estimates used
throughout the paper. Section~\ref{sec:nodal} proves the nodal-domain
bound and boundary strong unique continuation.
Section~\ref{sec:subcritical} establishes the subcritical asymptotics.
Section~\ref{sec:critical} identifies the
noninteger critical profiles and excludes integer critical order
under $C^{1,\mathrm{DMO}}$ regularity. Section~\ref{sec:supercritical}
proves supercritical quantization and exact trace and flux
asymptotics.

\section{Preliminaries}\label{sec:prelim}

\subsection{Boundary regularity and local coordinates}
Identify $\R^2$ with $\C$, writing $X$ for   points in $\Om$ and
$z=x+iy$ for conformal coordinates. Set
\[
 \begin{gathered}
 \HH=\{y>0\},\qquad B_R^+=B_R(0)\cap\HH,\qquad
 I_R=(-R,R)\times\{0\},\\
 S_R^+=\partial B_R(0)\cap\HH,\qquad
 Q_R^+=(-R,R)\times(0,R).
 \end{gathered}
\]
After translation and rotation, a $C^1$ boundary has the form
\begin{equation}\label{t134:eq:graph}
 q=0,\qquad \Om=\{(x,y):y>\psi(x)\}\quad  {local},
 \qquad\psi\in C^1,\quad\psi(0)=\psi'(0)=0.
\end{equation}
Define the rescaled domains and graph maps by
\begin{equation*}
 \Om_r=r^{-1}\Om,\qquad \psi_r(x)=r^{-1}\psi(rx),
 \qquad T_r(x,y)=(x,y+\psi_r(x)).
\end{equation*}
On each fixed bounded set,
\begin{equation*}
 T_r\longrightarrow\operatorname{id}\quad\text{in }C^1,
 \qquad \det DT_r=1.
\end{equation*}

The boundary is $C^{1,\mathrm{DMO}}$
near $q$ if the local graph representation \eqref{t134:eq:graph}
can be chosen on $J=(-r_0,r_0)$, for some $r_0>0$, with
\begin{equation*}
 \psi\in C^1(\overline J),\qquad
 \int_0^{r_0}\eta_{\psi'}(r)\,\frac{\dd r}{r}<\infty.
\end{equation*}
Here $J_r(x_0)=J\cap(x_0-r,x_0+r)$ and
\begin{equation}\label{t12:eq:dmo}
 \begin{split}
 \langle\psi'\rangle_{x_0,r}
   &=\frac1{|J_r(x_0)|}\int_{J_r(x_0)}\psi'(t)\dd t,\\
 \eta_{\psi'}(r)
   &=\sup_{x_0\in J}\frac1{|J_r(x_0)|}
       \int_{J_r(x_0)}
       |\psi'(t)-\langle\psi'\rangle_{x_0,r}|\dd t.
 \end{split}
\end{equation}

\subsection{Local growth and vanishing orders}

For a continuous function $v$ on the closure of a domain $E$ near
$x_0\in\overline E$, define its \emph{pointwise vanishing order} by
\begin{equation*}
 \ord^{\infty}_{x_0}v:=\sup\left\{b\ge0:
 \sup_{\overline E\cap B_r(x_0)}|v|=O(r^b)
 \text{ as }r\rightarrow0\right\}.
\end{equation*}

We first record a local $L^2$-$L^\infty$ estimate whose constant
is independent of the coefficient of the restoring boundary term.

\begin{lemma}\label{lem:moser}
Let $A$ be a measurable symmetric matrix on $Q_1^+$ such that, for
some $0<\theta\le\Theta<\infty$,
\[
 \theta|\xi|^2\le\xi^TA(X)\xi,\qquad
 |A(X)\xi|\le\Theta|\xi|
 \quad\text{for a.e. }X\text{ and every }\xi\in\R^2.
\]
Let $a,b\in L^\infty(I_1)$ satisfy $0\le a\le1$ and $b\ge0$.
Suppose $v\in H^1(Q_1^+)\cap L^\infty(Q_1^+)$ is real and
\begin{equation}\label{eq:graph-weak}
 \int_{Q_1^+}A\nabla v\cdot\nabla\zeta\dd X
 =\int_{I_1}(av-b|v|^{p-1}v)\zeta\dd s
\end{equation}
for every $\zeta\in H^1(Q_1^+)$ supported away from
$\partial Q_1^+\setminus I_1$. Then
\begin{equation}\label{eq:moser}
 \|v\|_{L^\infty(Q_{1/2}^+)}
 \le C(\theta,\Theta)\|v\|_{L^2(Q_1^+)}.
\end{equation}
The constant is independent of $p$, $b$, and the a priori supremum of
$v$.
\end{lemma}

\begin{proof}
For $m\ge2$, test \eqref{eq:graph-weak} with
$\eta^2|v|^{m-2}v$, where $\eta$ vanishes near the artificial
boundary, and set $w=|v|^{m/2}$, $W=\eta w$. The restoring term
is nonpositive. Ellipticity and Young's inequality therefore give
\[
 \int|\nabla W|^2
 \le C\int w^2|\nabla\eta|^2+Cm\int_{I_1}W^2.
\]
The trace inequality
\[
 \int_{I_1}W^2\le\varepsilon\int|\nabla W|^2
                  +\varepsilon^{-1}\int W^2
\]
with $\varepsilon=c/m$ absorbs the boundary term and yields
\[
 \int|\nabla W|^2
 \le C\int\bigl(|\nabla\eta|^2+m^2\eta^2\bigr)|v|^m.
\]
The two-dimensional $H^1$--$L^4$ inequality now implies, for
$1/2\le s<t\le1$,
\begin{equation*}
 \|v\|_{L^{2m}(Q_s^+)}
 \le\left(\frac{Cm^2}{(t-s)^2}\right)^{1/m}
       \|v\|_{L^m(Q_t^+)}.
\end{equation*}
Iteration with $m_j=2^{j+1}$ and $t_j=1/2+2^{-j-1}$ proves
\eqref{eq:moser}.
\end{proof}

\begin{proposition}[Boundary estimate in the original coordinates]\label{prop:physical}\label{t12:lem:physical}
Let $q$ be a Lipschitz boundary point, and let $u$ satisfy the local
problem and regularity assumptions
\eqref{paper:eq:local-problem}--\eqref{paper:eq:local-class}. There are constants
$c,C_1,C_2,r_*>0$ such that
\begin{equation}\label{eq:physical}
 \|u\|_{L^\infty(\Om\cap B_{cr}(q))}
 \le\frac{C_1}{r}
 \|u\|_{L^2(\Om\cap B_{C_2r}(q))},
 \qquad 0<r<r_*.
\end{equation}
The pointwise and bulk $L^2$-vanishing orders agree:
\begin{equation}\label{t12:eq:pointwise-equality}
 \ord^{\infty}_qu=\ord_q u.
\end{equation}
In particular, if $\ord_q u=+\infty$, then for every $N\in\N$
there are $C_N,r_N>0$ such that
\begin{equation}\label{eq:physical-flat}
 |u(X)|\le C_N|X-q|^N,
 \qquad X\in\overline\Om\cap B_{r_N}(q).
\end{equation}
\end{proposition}

\begin{proof}
In graph coordinates $T(s,t)=(s,\psi(s)+t)$, with $q=0$ and
$L=\Lip\psi$, the function $v=u\circ T$ satisfies
\begin{equation*}
 \int A\nabla v\cdot\nabla\zeta
 =\int J(s)(\lambda v-|v|^{p-1}v)\zeta,
\end{equation*}
where $\det DT=1$ and
\begin{equation}\label{eq:matrix}
 A=(DT)^{-1}(DT)^{-T}
 =\begin{pmatrix}1&-\psi'\\-\psi'&1+(\psi')^2\end{pmatrix},
 \qquad J=\sqrt{1+(\psi')^2}.
\end{equation}
The ellipticity constants and $\|J\|_\infty$ depend only on $L$.
For $v_r(Z)=v(rZ)$, the boundary coefficients in
Lemma~\ref{lem:moser} are $a_r=\lambda rJ(r\,\cdot)$ and
$b_r=rJ(r\,\cdot)$. Taking $r$ small enough that $a_r\le1$ gives
\begin{equation*}
 \|u\circ T\|_{L^\infty(Q_{r/2}^+)}
 \le C_Lr^{-1}\|u\circ T\|_{L^2(Q_r^+)}.
\end{equation*}
The bi-Lipschitz inclusions
$\Om\cap B_{cr}\subset T(Q_{r/2}^+)$ and
$T(Q_r^+)\subset\Om\cap B_{C_2r}$, together with $\det DT=1$,
prove \eqref{eq:physical}.

A pointwise bound of order $b$ gives the volume bound of order $b$.
Conversely, for every $0\le b<\ord_q u$, the defining volume bound
and \eqref{eq:physical} give $\|u\|_{L^\infty(\Om\cap B_r(q))}
=O(r^b)$. Taking suprema proves
\eqref{t12:eq:pointwise-equality}, including zero and infinite
order. In the latter case the same estimate for every integer $N$,
followed by continuity up to the boundary, gives
\eqref{eq:physical-flat}.
\end{proof}

\subsection{Conformal maps and the boundary equation}

A local conformal map at a boundary point $q$ is an injective holomorphic map
$\Phi:B_R^+\to\Om$ extending homeomorphically to the closed half-disk,
with $\Phi(0)=q$, the diameter mapped to a boundary arc, and
$\Om\cap B_\varepsilon(q)$ contained in its image for some $\varepsilon>0$.

The following is the local form of the boundary H\"older estimate
\cite[Lemma 1]{Warschawski}, together with the Riemann mapping and
Carath\'eodory theorems \cite{Pom}. The localization uses a Jordan
graph subdomain agreeing with $\Om$ near $q$; a Lipschitz graph
satisfies the local chord--arc hypothesis of that estimate.

\begin{lemma}[Conformal maps at Lipschitz boundaries]
\label{lem:chart}
At every Lipschitz boundary point $q$, there exist $r_0>0$ and a conformal map
$\Phi:B_{r_0}^+\to\Om$, where
$B_{r_0}^+=B_{r_0}(0)\cap\{\operatorname{Im}z>0\}$, with the
following properties. The map extends to a homeomorphism of its closure
onto its image, $\Phi(0)=q$, and the image of the diameter is a boundary
arc containing $q$ in its relative interior. For some $\alpha>0$,
\begin{equation}\label{eq:holder-chart}
 |\Phi(z)-q|\le C|z|^\alpha,
 \qquad z\in\overline{B_{r_0}^+}.
\end{equation}
No simple connectivity assumption on $\Om$ is required.
\end{lemma}

\begin{corollary}[Flatness in conformal coordinates]\label{cor:flat}
If $\ord_q u=+\infty$ and $U=u\circ\Phi$, then $U$ is harmonic,
continuous on the closed upper half-disk, and
\begin{equation}\label{eq:U-flat}
 |U(z)|\le C_M|z|^M\quad\text{near }0
 \quad\text{for every }M\in\N.
\end{equation}
Moreover, for every $M\in\N$,
\begin{equation}\label{eq:derivative-flat}
 |U_x(0,y)|\le C_M y^M\qquad(0<y\text{ sufficiently small}).
\end{equation}
\end{corollary}

\begin{proof}
Choose an integer $N$ with $\alpha N\ge M$. Equations
\eqref{eq:physical-flat} and \eqref{eq:holder-chart} give
$|U(z)|\le C_N|\Phi(z)-q|^N\le C'_N|z|^{\alpha N}$.
For \eqref{eq:derivative-flat}, use the interior gradient estimate in
$B_{y/2}((0,y))\subset B_{r_0}^+$:
\[
 |\nabla U(0,y)|\le\frac{C}{y}
 \sup_{B_{y/2}((0,y))}|U|\le C_M y^M,
\]
where \eqref{eq:U-flat} is applied with exponent $M+1$.
\end{proof}

\begin{remark}\label{rem:order-distortion}
For a general Lipschitz boundary, the preceding argument preserves
\emph{infinite pointwise flatness}. To compare finite bulk
$L^2$-vanishing orders, we use the additional control of the
conformal map provided by $C^1$ regularity below.
\end{remark}

\begin{lemma}[Conformal maps at $C^1$ boundaries]
\label{t12:lem:c1-chart}
Fix a local conformal map $\Phi$ at a $C^1$ boundary point
$q$. After restriction, the following properties hold.
\begin{enumerate}
\item[\textup{(a)}] The boundary parametrization is locally absolutely
continuous. Its derivative agrees almost everywhere with the
nontangential limit of the holomorphic derivative, and
\begin{equation}\label{t12:eq:arclength}
 \dd S=\rho(t)\dd t,\qquad \rho(t)=|\Phi'(t)|.
\end{equation}
For every finite $Q>1$, there exists $R_Q>0$ such that
\begin{equation}\label{t12:eq:rho-lq}
 \rho\in L^Q(-R_Q,R_Q).
\end{equation}
\item[\textup{(b)}] For every $0<\alpha<1$, there are constants and
smaller neighborhoods such that
\begin{equation}\label{t12:eq:two-holder}
 |\Phi(z)-q|\le C_\alpha|z|^\alpha,
 \qquad
 |\Phi^{-1}(X)|\le C_\alpha|X-q|^\alpha.
\end{equation}
The estimates hold up to the respective boundaries, for the same
fixed conformal map.
\end{enumerate}
Consequently, for any local solution $u$ of
\eqref{paper:eq:local-problem}--\eqref{paper:eq:local-class},
\begin{equation}\label{t12:eq:order-transfer}
\ord^{\infty}_ {0}(u\circ\Phi)=\ord^{\infty}_qu=\ord_q u\in[0,+\infty].
\end{equation}
\end{lemma}

\begin{proof}
After closing a smaller boundary arc by a $C^1$ cap, the
rectifiable-boundary theorem \cite{Pom} gives
\eqref{t12:eq:arclength}, while Lindel\"of's theorem
\cite{Kalaj} gives continuity
of $\arg\Phi'$ up to the diameter. These local conclusions transfer
to any fixed $\Phi$ by Schwarz reflection of the transition maps
and their inverses, whose boundary derivatives are nonzero.

Fix $Q>1$. After rotation and restriction,
$|\arg(e^{-i\theta_0}\Phi')|<\pi/(4Q)$. For a compatible logarithm,
\[
 H=\Rea\exp\bigl(Q\log(e^{-i\theta_0}\Phi')\bigr)>0,
 \qquad |\Phi'|^Q\le\sqrt2\,H.
\]
A positive harmonic function has integrable boundary trace with
respect to harmonic measure, whose density is bounded above and
below on smaller diameter intervals. Fatou's lemma gives
\eqref{t12:eq:rho-lq}.

For \eqref{t12:eq:two-holder}, normalize the tangent and reflect
across the graph using $R_\psi(x,y)=(x,2\psi(x)-y)$.
The extension and its inverse are locally $K(L)$-quasiconformal,
where $L=\|\psi'\|_\infty$ and $K(L)=(\sqrt{1+L^2}+L)^2\to1$
as the patch shrinks. Mori's theorem \cite{Mori} therefore gives
both estimates for every $\alpha<1$.

Composing the two estimates gives
$\ord^{\infty}_{0}(u\circ\Phi)\ge\alpha\ord^{\infty}_q(u)$ and
$\ord^{\infty}_q(u)\ge\alpha\ord^{\infty}_0(u\circ\Phi)$.
Letting $\alpha\rightarrow1$ and using Proposition~\ref{prop:physical}
proves \eqref{t12:eq:order-transfer}, also for infinite order.
\end{proof}

We use the local conformal diffeomorphism constructed by Vita
\cite{Vita}, taking its inverse as the local
conformal map. The assertion for any fixed conformal map follows from the analytic
transition maps described above.

\begin{lemma}[Nondegenerate conformal maps under DMO regularity]
\label{t12:lem:dmo-chart}\label{t134:lem:chart}\label{t15:lem:dmo-chart}
If $\partial\Om$ is $C^{1,\mathrm{DMO}}$ near $q$, there is a local
conformal map $\Phi$ that is a $C^1$ diffeomorphism up
to the boundary  and satisfies 
\begin{equation}\label{t12:eq:chart-nondegenerate}
 a_\Phi:=\Phi'(0)\ne0,\qquad
 0<c\le|\Phi'(z)|\le C
 \quad\text{for all }z\in\overline{B_R^+}.
\end{equation}
Every other
fixed local conformal map has the same regularity after
restriction. 
\end{lemma}

Under the hypotheses of Lemma~\ref{t12:lem:dmo-chart}, set
\[
 a_\Phi=\Phi'(0),\qquad a_0=|a_\Phi|,\qquad e=a_\Phi/a_0.
\]
Here $e$ and $ie$ are the oriented unit tangent and inward unit
normal at $q$.

\begin{lemma}[The transformed boundary equation]\label{lem:conformal-equation}
Let $q$ be a $C^1$ boundary point, let $u$ satisfy
\eqref{paper:eq:local-problem}--\eqref{paper:eq:local-class}, and let $\Phi$ be any fixed
local conformal map at $q$. After restriction, $U=u\circ\Phi$ is
harmonic and belongs to $H^1(B_R^+)\cap C(\overline{B_R^+})$.
Its outward normal trace on $I_R$ is
$\rho(t)(\lambda U-|U|^{p-1}U)$, with $\rho=|\Phi'|$ as in
Lemma \ref{t12:lem:c1-chart}.
\end{lemma}

\begin{proof}
Conformal invariance of harmonicity and Dirichlet energy, together
with continuity, gives $U\in H^1\cap C$ on smaller half-disks.
For a smooth test $\zeta$ supported away from the semicircle,
$\zeta\circ\Phi^{-1}$ has finite energy and extends by zero across
the artificial arc to an admissible test in the original coordinates.
Green's formula for \eqref{paper:eq:local-problem}, followed by
conformal change of variables and \eqref{t12:eq:arclength}, gives
\begin{equation}\label{t12:eq:pullback}
 \int_{B_R^+}\nabla U\cdot\nabla\zeta\dd z
 =\int_{-R}^{R}g(t)\zeta(t,0)\dd t,
 \qquad g=\rho(\lambda U-|U|^{p-1}U).
\end{equation}
Hence the outward normal trace is $-U_y=g$; absolute continuity of the
boundary parametrization excludes a singular boundary measure.
\end{proof}

\subsection{Local boundary estimates}

\begin{lemma}
\label{t12:lem:potential}
Let $V\in H^1(B_R^+)\cap C(\overline{B_R^+})$ be real and harmonic,
with outward normal trace $h$ on $I_R$. Suppose that
\begin{equation*}
 |h(t)|\le\omega(t)|t|^s,\qquad
 \omega\ge0,\quad \omega\in L^Q(-R,R),\quad
 2\le Q\le\infty,\quad s\ge0.
\end{equation*}
Use $1/Q=0$ when $Q=\infty$, and assume
\begin{equation*}
 \beta:=s+1-\frac1Q>0,\qquad \beta\notin\N.
\end{equation*}
Set $M=\lfloor\beta\rfloor$. Then there are real coefficients
$c_0,\ldots,c_M$ such that
\begin{equation}\label{t12:eq:potential-expansion}
 V(z)=\sum_{j=0}^{M}c_j\operatorname{Re}(z^j)
       +O(|z|^\beta)
\end{equation}
uniformly as $z\to0$ in the closed upper half-disk.
\end{lemma}

\begin{proof}
Choose $\chi\in C_c^\infty(-R,R)$ equal to one near zero, extend
$g=\chi h$ by zero, and set
\begin{equation*}
 S(z)=-\frac1\pi\int_\R g(t)\log|z-t|\dd t.
\end{equation*}
Since $g\in L^2$, the logarithmic potential is locally continuous
up to the diameter. The identities $-S_y=P_y*g$ and
$S_x=\pm\mathcal H(P_y*g)$, with the $L^2$ Poisson and
Hilbert-transform bounds, imply $S\in H^1_{\rm loc}$ up to the
diameter and outward normal trace $g$. Hence $V-S$ has an even
harmonic reflection near zero.

Let $Q'$ be conjugate to $Q$. Since $j\le M<\beta$, all moments
in
\[
 b_0=-\frac1\pi\int g(t)\log|t|\dd t,\qquad
 b_j=\frac1{\pi j}\int\frac{g(t)}{t^j}\dd t
 \quad(1\le j\le M)
\]
converge absolutely since $(s-j)Q'>-1$. Subtract
$\sum_{j=0}^M b_j\Rea(z^j)$ from $S$, put $r=|z|$, and split
the integral at $|t|=2r$. Scaling $t=r\tau$ and H\"older's inequality bound the
near part by $Cr^{s+1/Q'}=Cr^\beta$, using
\[
 \sup_{\substack{|\zeta|=1\\\Ima\zeta\ge0}}
 \left\||\tau|^s\log\left|\frac{\zeta-\tau}{\tau}\right|
 \right\|_{L^{Q'}(-2,2)}<\infty
\]
and the same moment integrability for the subtracted powers.
For the far part, Taylor's formula for the logarithm gives
\[
 C r^{M+1}\int_{2r<|t|<R}\omega(t)|t|^{s-M-1}\dd t
 \le C r^{M+1}
       \left(\int_{2r}^{R}t^{(s-M-1)Q'}\dd t\right)^{1/Q'}
 \le Cr^\beta,
\]
because $(s-M-1)Q'+1=Q'(\beta-M-1)<0$. These estimates include
$Q=\infty$ and are uniform up to the diameter. Finally, the even
harmonic function $V-S$ has a convergent expansion in $\Rea(z^j)$;
its truncation after $M$ has error $O(r^{M+1})$. Combining the
two expansions proves \eqref{t12:eq:potential-expansion}.
\end{proof}

\begin{lemma}[Boundary differentiability under DMO regularity]\label{t15:lem:boundary-c1}
Every locally $H^1\cap C$ solution of the nonlinear Robin problem on a
$C^{1,\mathrm{DMO}}$ boundary is locally $C^1$ up to that boundary.
\end{lemma}
\begin{proof}
With $V=u\circ T$ in graph coordinates and $A,J$ as in
\eqref{eq:matrix}, set
\[
 b(x)=J(x)\{\lambda V(x,0)-|V(x,0)|^{p-1}V(x,0)\},\qquad
 \mathbf G(x,y)=(0,-b(x)).
\]
The transformed problem is
\[
 \operatorname{div}(A\nabla V)=\operatorname{div}\mathbf G,
 \qquad A\nabla V\cdot e_2=\mathbf G\cdot e_2
 \quad\text{on the diameter}.
\]
Since $b$ is bounded, the reflected De Giorgi--Nash estimate gives
$V\in C^{0,\alpha}$ up to the diameter, hence
$\lambda V-|V|^{p-1}V\in C^{0,p\alpha}$ there.
Lipschitz composition and the product property
\cite{DongLeeKim} show that $A,J,b$
have DMO; constant extension in $y$ gives the same property for
$\mathbf G$. The flat conormal estimate
\cite{DongLeeKim} yields
$V\in C^1$, and composition with $T^{-1}$ proves the claim.
\end{proof}

\section{Nodal domains and boundary unique continuation}\label{sec:nodal}

\subsection{Strict comparison with the Steklov spectrum}

The spectral comparison holds for real $H^1$ solutions on bounded
connected Lipschitz domains $\Om\subset\R^n$, $n\ge2$, with the boundary
equation imposed on all of $\partial\Om$.

For each nodal domain $D$ of $u$, write $u_D:=u\ind_D$.
These nodal restrictions satisfy
\begin{equation}\label{eq:restriction}
 u_D\in H^1(\Om),\qquad
 \nabla u_D=\ind_D\nabla u\quad\text{a.e. in }\Om.
\end{equation}
Their Sobolev traces satisfy
\begin{equation}\label{eq:trace-support}
 u_D=u\quad\text{a.e. on }\{x\in\partial\Om:u_D(x)\ne0\}.
\end{equation}
For distinct nodal domains $D,E$, we also have
\begin{equation*}
 \nabla u_D\cdot\nabla u_E=0\quad\text{a.e. in }\Om,
 \qquad u_Du_E=0\quad\text{a.e. on }\partial\Om.
\end{equation*}

\begin{proposition}\label{prop:nodal}
Let $\Om\subset\R^n$ ($n\ge2$) be a bounded connected Lipschitz domain,
and let $0<p<1$ and $\lambda>0$. Suppose that
$u\in H^1(\Om)\setminus\{0\}$ is real and solves \eqref{eq:problem}.
Then $d(u)\le N_S^{<}(\lambda)$.
\end{proposition}

\begin{proof}
Apply Green's formula for \eqref{eq:problem} with the admissible
test function $u_D\in H^1(\Om)$. Using
\eqref{eq:restriction}--\eqref{eq:trace-support} gives
\begin{equation*}
 E_D=\lambda B_D-Q_D,
\end{equation*}
where $E_D=\int_\Om|\nabla u_D|^2$,
$B_D=\int_{\partial\Om}|u_D|^2$, and
$Q_D=\int_{\partial\Om}|u_D|^{p+1}$.
If $B_D=0$, then $E_D=0$ and connectedness and the zero trace force
$u_D=0$, a contradiction. Thus
\begin{equation*}
 B_D>0,\qquad Q_D>0,\qquad E_D/B_D=\lambda-Q_D/B_D<\lambda.
\end{equation*}
For any finite family $D_1,\ldots,D_k$, let
$V_k=\Span\{u_{D_1},\ldots,u_{D_k}\}$.
Disjointness of the gradients and boundary traces gives, for
$v=\sum_i a_i u_{D_i}$,
\begin{equation*}
 \int_\Om|\nabla v|^2=\sum_i a_i^2E_i,
 \qquad \int_{\partial\Om}|v|^2=\sum_i a_i^2B_i.
\end{equation*}
Hence $\dim(V_k|_{\partial\Om})=k$ and, with $\delta_k=\min_i Q_i/B_i>0$,
\begin{equation*}
 \sup_{v\in V_k\setminus\{0\}}
 \frac{\int_\Om|\nabla v|^2}{\int_{\partial\Om}|v|^2}
 \le\lambda-\delta_k<\lambda.
\end{equation*}
Harmonic extension preserves this $k$-dimensional trace space and
does not increase Dirichlet energy. The Steklov min--max principle
therefore yields
\begin{equation*}
 \sigma_{k-1}(\Om)\le\lambda-\delta_k<\lambda.
\end{equation*}
Thus $k\le N_S^{<}(\lambda)$ for every finite family, which proves
the assertion without assuming that $d(u)$ is finite.
\end{proof}

\subsection{Boundary sign rigidity and planar topology}
The following lemma follows from Baouendi and Rothschild~
\cite[Corollary~2]{BR}.

\begin{lemma}\label{lem:sign}
Let $h$ be real, harmonic in an upper half-disk, continuous up to its
diameter, and pointwise flat at the origin in the sense of
\eqref{eq:U-flat}. If, on a neighborhood of zero in the diameter,
either $h(x,0)$ or $xh(x,0)$ does not change sign, then $h\equiv0$.
\end{lemma}

The following planar separation lemma rules out alternating boundary
contacts by two disjoint connected open sets.

\begin{lemma}\label{lem:crossing}
Let $\Gamma$ be a boundary component of a bounded connected planar
Lipschitz domain $\Om$. Let $P,N$ be disjoint connected open subsets
of $\Om$. It is impossible to have four distinct points
$a,b,c,d\in\Gamma$ in cyclic order such that $P$ has interior
collars at $a,c$ and $N$ has interior collars at $b,d$.
\end{lemma}

\begin{proof}
Each component of a planar Lipschitz boundary is a Jordan curve.
On the Riemann sphere, $\Om$ lies in one of the two topological
disks bounded by $\Gamma$, say $\mathcal D$.
Path connectedness of $P$ and the collars at $a,c$ give a simple
crosscut of $\mathcal D$ from $a$ to $c$ with interior in $P$.
By the Jordan crosscut theorem, the collars at $b,d$ lie on
opposite sides. A path between them in $N$ must meet the crosscut,
contrary to $P\cap N=\varnothing$.
\end{proof}

\begin{proposition}\label{prop:infinite-nodal}
Let $\Om\subset\R^2$ be a bounded connected Lipschitz domain,
$q\in\partial\Om$, and let $u\in H^1(\Om)\cap C(\overline\Om)$
be real, harmonic and not identically zero. Suppose that, on a
boundary neighborhood of $q$,
\[
 \partial_\nu u=\lambda u-|u|^{p-1}u,
 \qquad 0<p<1,\quad\lambda\ge0,
\]
in the sense of normal traces.
No condition is imposed on the remainder of $\partial\Om$.
If $\ord_q u=+\infty$, then $d(u)=\infty$.
\end{proposition}

\begin{proof}
For the conformal map in Lemma~\ref{lem:chart},
$U=u\circ\Phi$ is pointwise flat by Corollary~\ref{cor:flat},
and is nonzero by interior unique continuation.
If its trace does not change sign on each side of zero, then
either $U(x,0)$ or $xU(x,0)$ does not change sign, contradicting
Lemma~\ref{lem:sign}. On at least one side, therefore, a monotone
sequence approaching zero has strictly alternating trace signs.
Continuity gives disjoint boundary intervals and connected
interior collars with these signs. Their images under $\Phi$
lie, in the same or reversed order, on one boundary component.

If there were finitely many nodal domains, two of the consecutive
positive--negative pairs would have the same ordered nodal labels
$(P,N)$. Their four collars would occur in cyclic order
$P,N,P,N$, contradicting Lemma~\ref{lem:crossing}.
\end{proof}

\begin{proof}[\normalfont\bfseries Proof of Theorem \ref{thm:main}]
Proposition \ref{prop:nodal} proves \eqref{eq:nodal-main}. If
\eqref{eq:finite-main} failed at a boundary point, Proposition
\ref{prop:infinite-nodal} would imply infinitely many nodal domains,
contradicting \eqref{eq:nodal-main}.
\end{proof}

\begin{proof}[\normalfont\bfseries Proof of Proposition \ref{t15:local-sucp}]
Suppose that the germ is nonzero. In local graph coordinates at
$q=0$, continuity allows us to arrange
\begin{equation*}
 \lambda|u|^{1-p}\le\tfrac12.
\end{equation*}
The trace cannot vanish on an open boundary arc: the Robin law
would also give zero normal trace there. Green's formula then
shows that zero extension across the Lipschitz graph is
distributionally harmonic, hence harmonic by Weyl's lemma.
Interior unique continuation would make the germ zero.

Choose $a<0<b$ with nonzero trace at $(a,\psi(a))$ and
$(b,\psi(b))$, and $h>\max_{[a,b]}\psi$ with
$u(a,h)u(b,h)\ne0$, all within this neighborhood. Such an $h$
exists because the restrictions to the two vertical lines are
nonzero real analytic functions. Let
\[
 D=\{(x,y):a<x<b,\ \psi(x)<y<h\},
\]
with bottom boundary $\Gamma$ and remaining three sides $\Sigma$.
On each straight side the trace is analytic away from the bottom
endpoints; it is nonzero near those endpoints and at the upper
corners. Thus $\Sigma\cap\{u\ne0\}$ has finitely many connected
components.

If a nodal restriction $u_E$ in $D$ had zero trace on $\Sigma$,
its zero extension across $\Sigma$ would be an admissible local
test function by \eqref{eq:restriction}. It would give
\[
 \int_D|\nabla u_E|^2+\int_\Gamma|u_E|^{p+1}
 =\lambda\int_\Gamma|u_E|^2
 \le\tfrac12\int_\Gamma|u_E|^{p+1},
\]
forcing $u_E=0$. Every nodal domain therefore has nonzero trace
on $\Sigma$. Along each connected component of
$\Sigma\cap\{u\ne0\}$, continuity and interior collars assign a
unique nodal domain, locally constantly and hence constantly.
This gives only finitely many nodal domains in $D$.
Since $D$ agrees with $\Om$ near zero, its solution has infinite
vanishing order there. Proposition~\ref{prop:infinite-nodal} gives
the contradiction.
\end{proof}

\section{Subcritical boundary zeros}\label{sec:subcritical}

In this section, we prove Theorem~\ref{t12:thm:main}.
We first treat the $C^1$ case, then use Dini mean oscillation to obtain
the endpoint remainder estimate in conformal coordinates and a nondegenerate leading profile
in the original coordinates.

\begin{proof}[\normalfont\bfseries Proof of Theorem~\ref{t12:thm:main}]
\textup{(i)} Write $U=u\circ\Phi$ and $g=\rho(\lambda U-|U|^{p-1}U)$ for its
outward Neumann datum. Proposition~\ref{prop:physical} and
Lemma~\ref{t12:lem:c1-chart} give
\begin{equation*}
 \ord^{\infty}_0 U=k,\qquad
 |U(z)|\le C_\varepsilon|z|^{k-\varepsilon}
 \quad(0<\varepsilon<k).
\end{equation*}
Choose $\varepsilon>0$ small and a finite $Q\ge2$ large enough that
\begin{equation*}
 \beta=1+p(k-\varepsilon)-1/Q>k,\qquad\beta\notin\N.
\end{equation*}
Since $|g(t)|\le C_\varepsilon\rho(t)|t|^{p(k-\varepsilon)}$
and $\rho\in L^Q$ locally, Lemma~\ref{t12:lem:potential} gives
\begin{equation*}
 U(z)=\sum_{j=0}^{\lfloor\beta\rfloor}c_j\Rea z^j
       +O(|z|^\beta).
\end{equation*}
Some coefficient is nonzero, since otherwise the order would
exceed $k$. The least nonzero degree equals $k$: evaluating on a
ray where that homogeneous term is nonzero excludes larger
orders. Thus $k\in\N$, and $k\ge1$ because $U(0)=0$.
As $k<\beta<1+pk<k+1$, only $A=c_k\ne0$ remains, and
\begin{equation*}
 |U(z)|\le C|z|^k.
\end{equation*}
Consequently,
\begin{equation}\label{t12:eq:improved-data}
 |g(t)|\le C\rho(t)|t|^{pk}.
\end{equation}
For any $0<\delta<\delta_*$, choose $1/Q<\delta_*-\delta$.
A second application of Lemma~\ref{t12:lem:potential} has exponent
$k+\delta<1+pk-1/Q<k+1$. Comparing coefficients gives
\begin{equation}\label{t12:eq:conformal-expansion}
 U(z)=A\operatorname{Re}(z^k)+O(|z|^{k+\delta}),
 \qquad 0<\delta<\delta_*,
\end{equation}
uniformly as $z\to0$ in the closed half-disk. In particular,
\begin{equation*}
 A=\lim_{t\rightarrow0}t^{-k}U(t,0),
\end{equation*}
so $A$ is unique.

The expansion gives locally uniform convergence of
$W_r(z)=r^{-k}U(rz)-A\Rea z^k$ to zero on the closed half-plane.
Its Neumann datum is $h_r(t)=r^{1-k}g(rt)$, and for $1/Q<\delta_*$,
\begin{equation*}
 \|h_r\|_{L^Q(-2L,2L)}
 \le C_Lr^{\delta_*-1/Q}
       \|\rho\|_{L^Q(-2Lr,2Lr)}\longrightarrow0.
\end{equation*}
Testing with $\eta^2W_r$, where $\eta=1$ on $B_L$ and is supported
in $B_{2L}$, gives
\begin{equation*}
 \int\eta^2|\nabla W_r|^2
 \le C\int|W_r|^2|\nabla\eta|^2
       +C\int_{-2L}^{2L}|h_rW_r|\longrightarrow0.
\end{equation*}
This proves strong $H^1$ convergence. Finally, squaring the
diameter expansion and integrating yields
\[
 \int_{-r}^rU(t,0)^2\dd t
 =\frac{2A^2}{2k+1}r^{2k+1}+O(r^{2k+\delta+1}),
\]
which is \eqref{t12:eq:trace}.

\smallskip
\noindent\textup{(ii)} Under DMO, Lemma~\ref{t12:lem:dmo-chart} gives
the regularity of the conformal map and \eqref{t12:eq:chart-nondegenerate}.
In particular, $\rho$ is bounded, so \eqref{t12:eq:improved-data}
gives $|g(t)|\le C|t|^{pk}$.
Since
\begin{equation*}
 k-1<pk<k,\qquad k<1+pk<k+1,
\end{equation*}
Lemma~\ref{t12:lem:potential} with $Q=\infty$ gives the endpoint
expansion with the same $A$.

Writing $a_\Phi=\Phi'(0)\ne0$, the $C^1$ conformal map and its inverse satisfy
\begin{equation*}
 \Phi(z)=q+a_\Phi z+o(|z|),\qquad
 \Phi^{-1}(q+w)=w/a_\Phi+o(|w|),
\end{equation*}
uniformly up to the boundary. Substituting in the conformal
expansion gives \eqref{t12:eq:physical-expansion}, since
$|z^k-\zeta^k|\le k\max\{|z|,|\zeta|\}^{k-1}|z-\zeta|$.
The tangent direction is $e=a_\Phi/|a_\Phi|$, and $ie$ points
inward. Thus $P_k$ is a nonzero harmonic polynomial with zero
normal derivative on $\partial H_{a_\Phi}$. Restriction to
$w=e(s+i\psi(s))$, where $\psi(s)=o(|s|)$, gives the stated
boundary trace. Homogeneity gives the uniform rescaled limit.

The $C^1$ graph implies
$\ind_{\Om_r}\to\ind_{H_{a_\Phi}}$ almost everywhere.
Uniform convergence and dominated convergence then give
\[
 r^{-2k-2}\int_{\Om\cap B_r(q)}u^2
 \longrightarrow\int_{H_{a_\Phi}\cap B_1}P_k^2
 =\frac{A^2}{|a_\Phi|^{2k}}
   \int_0^1s^{2k+1}\dd s\int_0^\pi\cos^2(k\theta)\dd\theta
 =\frac{\pi A^2}{4(k+1)|a_\Phi|^{2k}}.
\]
\end{proof}

\section{Critical boundary zeros}\label{sec:critical}

In this section, we prove Theorems~\ref{t134:ni:thm:main}
and~\ref{thm:integer-critical-order} on critical boundary zeros.
For noninteger $\kappa$, we work directly on $\Om$,
using compactness under graph rescaling to identify the unique
asymptotic profile under $C^1$ regularity. For integer $\kappa$,
Dini mean oscillation provides a nondegenerate conformal map,
which reduces the problem to a flat equation with a continuous
positive boundary weight; we then exclude vanishing of order $\kappa$
for this weighted problem.

We first record the flat Weiss identity for the limiting equations
used in both the noninteger and integer cases.

\begin{lemma}\label{t134:lem:flat-weiss}
Let $c>0$ and suppose that the real-valued function $V$
belongs to $H^1(B_r^+)\cap L^\infty(B_r^+)$ for every $0<r<R_0$
and satisfies
\[
 \Delta V=0\quad\text{in }B_{R_0}^+,\qquad
 -V_y=-c |V|^{p-1}V\quad\text{on }I_{R_0}.
\]
Then  the functional
\begin{equation*}
 W_c(V,r)=r^{-2\kappa}\left(
 \int_{B_r^+}|\nabla V|^2+
 \frac{2c}{p+1}\int_{I_r}|V|^{p+1}\right)
 -\kappa r^{-2\kappa-1}\int_{S_r^+}V^2
\end{equation*}
satisfies
\begin{equation}\label{t134:eq:weiss-general-derivative}
 \frac{d}{dr}W_c(V,r)=2r^{-2\kappa}\int_{S_r^+}
       \left(\partial_rV-\frac\kappa rV\right)^2\ge0, \qquad a.e.\qquad  0<r<R_0.
\end{equation}
In particular, if $W_c(V,\cdot)$ is constant on a nontrivial interval,
then $V$ is $\kappa$-homogeneous.
\end{lemma}

\begin{proof}
The bounded Neumann datum $h=-c|V|^{p-1}V$ has a cutoff logarithmic
potential with modulus $C\delta(1+|\log\delta|)$; the remainder
extends harmonically by even reflection. Hence $V\in C^{0,\alpha}$
locally up to the diameter for every $\alpha<1$, and
$h\in C^{0,\alpha p}$. Subtracting the density at the evaluation
point in the differentiated kernel gives the gradient modulus
\[
 C\int_0^{2\delta}t^{\alpha p-1}\dd t
 +C\delta\int_{2\delta}^Rt^{\alpha p-2}\dd t
 \le C\delta^{\alpha p}.
\]
Thus $V\in C^{1,\gamma}$ for every $\gamma<p$.

Set
\[
 B(r)=\int_{I_r}|V|^{p+1},\qquad
 H(r)=\int_{S_r^+}V^2,\qquad
 D(r)=\int_{B_r^+}|\nabla V|^2+\frac{2c}{p+1}B(r).
\]
Then $W_c(V,r)=r^{-2\kappa}D(r)-\kappa r^{-2\kappa-1}H(r)$.
Under the rescaling $V^{(r)}(Z)=r^{-\kappa}V(rZ)$, the Dirichlet,
flat-boundary, and semicircle integrals acquire, respectively, the factors
\[
 r^{2-2\kappa}r^{-2}=r^{-2\kappa},\qquad
 r^{-\kappa(p+1)}r^{-1}=r^{-2\kappa},\qquad
 r^{-2\kappa}r^{-1}=r^{-2\kappa-1},
\]
where $1+\kappa(p+1)=2\kappa$.

Since the outward normal on $I_r$ is $-e_y$, Green's identity gives
\[
 \int_{B_r^+}|\nabla V|^2
 =\int_{S_r^+}V\partial_rV-cB(r),\qquad
 D(r)=\int_{S_r^+}V\partial_rV+\frac{c(1-p)}{p+1}B(r).
\]
Differentiating the volume and interval integrals, and writing
$H(r)=r\int_0^\pi V(r\cos\theta,r\sin\theta)^2\dd\theta$, yield
\begin{align*}
 B'(r)&=|V(r,0)|^{p+1}+|V(-r,0)|^{p+1},\\
 D'(r)&=\int_{S_r^+}|\nabla V|^2+\frac{2c}{p+1}B'(r),\\
 H'(r)&=\frac{H(r)}r+2\int_{S_r^+}V\partial_rV.
\end{align*}

To obtain the Pohozaev identity, consider
$\mathcal P=|\nabla V|^2X-2(X\cdot\nabla V)\nabla V$.
In dimension two,
$\operatorname{div}\mathcal P=-2(X\cdot\nabla V)\Delta V=0$.
Its boundary fluxes are
\[
 \mathcal P\cdot\nu=
 \begin{cases}
 r|\nabla V|^2-2r(\partial_rV)^2,&\text{on }S_r^+,\\
 2cx|V|^{p-1}V\partial_xV
 =\dfrac{2c}{p+1}x\partial_x(|V|^{p+1}),&\text{on }I_r.
 \end{cases}
\]
The endpoint terms on the diameter satisfy
\[
 \int_{-r}^r x\partial_x(|V(x,0)|^{p+1})\dd x
 =r\bigl(|V(r,0)|^{p+1}+|V(-r,0)|^{p+1}\bigr)-B(r)
 =rB'(r)-B(r).
\]
Thus the divergence theorem gives
\[
 0=r\int_{S_r^+}|\nabla V|^2
   -2r\int_{S_r^+}(\partial_rV)^2
   +\frac{2c}{p+1}\bigl(rB'(r)-B(r)\bigr),
\]
and hence
\[
 rD'(r)=2r\int_{S_r^+}(\partial_rV)^2+\frac{2c}{p+1}B(r).
\]
These identities follow by integration on truncated half-disks and
passage to the flat boundary using the established $C^1$ regularity.

We now differentiate both powers of $r$ in $W_c(V,r)$ and substitute
the identities for $D$, $D'$, and $H'$:
\begin{align*}
 \frac{d}{dr}W_c(V,r)
 ={}&r^{-2\kappa}D'(r)-2\kappa r^{-2\kappa-1}D(r)\\
 &-\kappa r^{-2\kappa-1}H'(r)
   +\kappa(2\kappa+1)r^{-2\kappa-2}H(r)\\
 ={}&2r^{-2\kappa}\int_{S_r^+}(\partial_rV)^2
   -4\kappa r^{-2\kappa-1}\int_{S_r^+}V\partial_rV\\
 &+2\kappa^2r^{-2\kappa-2}H(r)
   +\frac{2c}{p+1}\bigl[1-\kappa(1-p)\bigr]r^{-2\kappa-1}B(r)\\
 ={}&2r^{-2\kappa}\int_{S_r^+}
       \left(\partial_rV-\frac\kappa rV\right)^2.
\end{align*}
\end{proof}

\begin{lemma}
\label{t134:ni:lem:compactness}
On expanding neighborhoods of the origin, let
$\mathcal D_j=\{y>\psi_j(x)\}$, with
$\psi_j\to0$ in $C^1_{\mathrm{loc}}(\R)$, and set
$\mathcal T_j(x,y)=(x,y+\psi_j(x))$.
Let $v_j$ be real harmonic functions in $\mathcal D_j$, locally
$H^1$ up to the graph boundary, satisfying there in the sense of normal traces
\[
 \partial_{\nu_j}v_j
 =\rho_j(x)\bigl(\epsilon_jv_j-\mu_j|v_j|^{p-1}v_j\bigr),
 \qquad \epsilon_j\ge0,\quad\epsilon_j\to0,\quad\mu_j\ge0.
\]
Assume the measurable weights satisfy $0<c_0\le\rho_j\le c_1$
and $\rho_j\to\rho_0>0$ locally uniformly, with $\rho_0$ constant.
Suppose, for every fixed $R>0$ and all sufficiently large $j$, that
\[
 \int_{\mathcal D_j\cap B_R}v_j^2\leq C_R.
\]
Then $\widetilde v_j:=v_j\circ\mathcal T_j$ is uniformly bounded in
$H^1(B_R^+)\cap L^\infty(B_R^+)$ for each $R>0$, independently of $\mu_j$.
After passing to a subsequence, $\mu_j\to\mu_0\in[0,\infty]$ and
$\widetilde v_j\to V$ strongly in $H^1(B_R^+)$ for every $R>0$,
where $V$ is harmonic in $\HH$. Moreover:
\begin{enumerate}[label=\textup{(\roman*)}]
\item If $\mu_0<\infty$, the convergence is locally uniform up to
the flat boundary, and
\[
 -V_y=-\mu_0\rho_0|V|^{p-1}V\quad\text{on }\R.
\]
\item If $\mu_0=\infty$, then $V=0$ on $\R$ and, for every $R>0$,
\[
 \mu_j\int_{\partial\mathcal D_j\cap B_R}|v_j|^{p+1}\longrightarrow0.
\]
\end{enumerate}
For every $R>0$, the $L^2$ masses and Dirichlet energies over
$\mathcal D_j\cap B_R$ converge to those of $V$ over $B_R^+$.
\end{lemma}

\begin{proof}
Flattening on a fixed half-rectangle gives
\begin{equation}\label{t134:ni:eq:flattened}
 \int A_j\nabla\widetilde v_j\cdot\nabla\zeta
 +\mu_j\int b_j|\widetilde v_j|^{p-1}\widetilde v_j\zeta
 -\epsilon_j\int b_j\widetilde v_j\zeta=0,
\end{equation}
where $A_j$ is the matrix in \eqref{eq:matrix} with $\psi=\psi_j$
and $b_j=\rho_j\sqrt{1+\psi_j'^2}$.
Thus $A_j\to I$ and $b_j\to\rho_0$ uniformly, with uniform
ellipticity and positive weight bounds. Since $\det D\mathcal T_j=1$,
the assumed volume bounds control the flattened norms.
Testing with $\eta^2\widetilde v_j$ and absorbing the linear
boundary term by the trace inequality gives
\begin{equation}\label{t134:ni:eq:caccioppoli}
 \int\eta^2|\nabla\widetilde v_j|^2
 +\mu_j\int b_j\eta^2|\widetilde v_j|^{p+1}
 \le C\int(|\nabla\eta|^2+1)\widetilde v_j^2.
\end{equation}
The truncated power tests used in Lemma~\ref{lem:moser} similarly give
\begin{equation*}
 \|\widetilde v_j\|_{L^\infty(B_R^+)}
 \le C_R\|\widetilde v_j\|_{L^2(B_{2R}^+)}.
\end{equation*}
Both constants are independent of $\mu_j$, because its boundary
term is nonnegative and can be discarded. Rellich compactness
and compactness of the trace give a subsequence converging weakly
in local $H^1$ and strongly in local volume and trace $L^2$
to a harmonic function $V$. Passing to a further subsequence,
we may assume $\mu_j\to\mu_0\in[0,\infty]$.

If $\mu_j\to\mu_0<\infty$, strong trace convergence passes
\eqref{t134:ni:eq:flattened} to the limiting boundary equation.
Subtract it and test with $\eta^2(\widetilde v_j-V)$.
Monotonicity makes
$\mu_jb_j(|\widetilde v_j|^{p-1}\widetilde v_j-|V|^{p-1}V)(\widetilde v_j-V)$ nonnegative;
uniform coefficient convergence and Young's inequality control
the other terms. Strong volume convergence then proves strong
local $H^1$ convergence. For uniform boundary compactness, the
conormal data
$h_j=b_j(\epsilon_j\widetilde v_j-\mu_j|\widetilde v_j|^{p-1}\widetilde v_j)$
are locally bounded. With $F_j=(0,-h_j(x))$, the equation is
$\operatorname{div}(A_j\nabla\widetilde v_j)=\operatorname{div}F_j$.
Even reflection, with the corresponding parities for $A_j,F_j$,
gives an interior uniformly elliptic equation with bounded
divergence data. The interior H\"older estimate and
Arzel\`a--Ascoli yield locally uniform convergence.

If $\mu_j\to\infty$, \eqref{t134:ni:eq:caccioppoli} gives
$V=0$ on $\R$. Odd reflection makes $V$ smooth across the diameter;
write $q=\partial_\nu V$ there. Subtract its Green identity,
including the boundary term, from \eqref{t134:ni:eq:flattened}
and use $\eta^2w_j$ as the test function, where
$w_j=\widetilde v_j-V$.
The penalty is exactly
$\mu_j\int b_j\eta^2|\widetilde v_j|^{p+1}$.
The remaining boundary terms
\[
 \epsilon_j\int b_j\eta^2\widetilde v_j^2
       -\int\eta^2q\widetilde v_j
\]
tend to zero by trace convergence. The bulk error involving
$(A_j-I)\nabla V$ and the cutoff term are controlled by Young's
inequality and strong volume convergence. This proves both
strong $H^1$ convergence and vanishing of the weighted penalty;
$\rho_j\ge c_0$ gives its unweighted version. The argument retains
$\mu_jb_j$ exactly and requires no bound on
$\mu_j(b_j-\rho_0)$.

Finally, $\mathcal T_j\to\operatorname{id}$ in $C^1$ and the
pulled-back ball indicators converge almost everywhere.
Strong $L^2$ convergence of the functions and gradients gives
the asserted convergence of  volume and energy integrals.
\end{proof}

\subsection{Noninteger critical order on a \texorpdfstring{$C^1$}{C1} domain}
\label{t134:sec:noninteger}
Throughout this subsection, we work under the hypotheses of
Theorem~\ref{t134:ni:thm:main}, with $\ord_0u=\kappa\notin\N$.

In the applications below, we take
$\mathcal D_j=\Omega_{r_j}$,
$\psi_j(x)=r_j^{-1}\psi(r_jx)$,
$\mathcal T_j=T_{r_j}$, $\rho_j\equiv1$, and
$\epsilon_j=\lambda r_j\to0$.

Fix $r_0>0$ sufficiently small and set $t_0=-\log r_0$.
For $0<r\leq r_0$ and $t\geq t_0$, define
\begin{equation*}
 \mathcal M(r)=\left(r^{-2\kappa-2}\int_{\Omega\cap B_r}u^2\right)^{1/2},
 \qquad A(t)=\mathcal M(e^{-t}).
\end{equation*}
These functions are positive and continuous. By the definition of
the vanishing order in \eqref{eq:order}, the assumption
$\ord_0u=\kappa$ implies that, for every $\epsilon>0$, there exists
a constant $C_\epsilon>0$ such that
\begin{equation}\label{t134:ni:eq:order-facts}
 \begin{aligned}
  A(t)\leq C_\epsilon e^{\epsilon t}\qquad
          \text{for all }t\geq t_0,\qquad
  \limsup_{t\to\infty}e^{\epsilon t}A(t)=+\infty.
 \end{aligned}
\end{equation}

\begin{lemma}\label{t134:ni:lem:scale-selection}
Let $t_0\in\R$ and let $A:[t_0,\infty)\to(0,\infty)$ be continuous.
Suppose that, for every $\epsilon>0$, there is $C_\epsilon>0$ such that
\[
 A(t)\leq C_\epsilon e^{\epsilon t}\qquad(t\geq t_0).
\]
For every fixed $\delta>0$, the following statements hold.
\begin{enumerate}[label=\textup{(\roman*)}]
\item If $A$ is unbounded, there exist $s_j\to\infty$ and
$a_j=A(s_j)\to\infty$ such that
\begin{equation}\label{t134:ni:eq:envelope}
 A(t)\leq a_j e^{\delta|t-s_j|}\qquad(t\geq t_0).
\end{equation}
\item If $A(t)\to0$ as $t\to\infty$ and
\[
 \limsup_{t\to\infty}e^{\epsilon t}A(t)=+\infty
 \qquad\text{for every }\epsilon>0,
\]
then there exist $s_j\to\infty$ and $a_j=A(s_j)\to0$
satisfying \eqref{t134:ni:eq:envelope}.
\end{enumerate}
\end{lemma}

\begin{proof}
For $T\geq t_0$, set
\[
 F_T(t):=A(t)e^{-\delta|t-T|},\qquad t\geq t_0.
\]
By the assumed growth bound,
\[
 0<F_T(t)\leq C_{\delta/2}e^{\delta T-\delta t/2}
 \longrightarrow0\qquad(t\to\infty,\ t\geq T).
\]
Thus continuity gives a maximizer
$s(T)\in\operatorname*{arg\,max}_{t\geq t_0}F_T(t)$, and
\begin{align*}
 A(t)&\leq A(s(T))e^{\delta(|t-T|-|s(T)-T|)}\\
     &\leq A(s(T))e^{\delta|t-s(T)|},\qquad t\geq t_0.
\end{align*}
For \textup{(i)}, choose $T_j\to\infty$ with $A(T_j)\to\infty$
and set $s_j=s(T_j)$, $a_j=A(s_j)$. Then
\[
 a_j\geq F_{T_j}(s_j)\geq F_{T_j}(T_j)=A(T_j)\longrightarrow\infty.
\]
Since $A$ is bounded on compact intervals, $s_j\to\infty$;
the preceding estimate gives \eqref{t134:ni:eq:envelope}.

For \textup{(ii)}, the limsup assumption with $\epsilon=\delta/2$
gives a sequence $T_j\to\infty$ such that
\[
 A(T_j)\geq e^{-\delta T_j/2}.
\]
Again set $s_j=s(T_j)$ and $a_j=A(s_j)$. For each fixed $K>t_0$,
write $M_K=\max_{[t_0,K]}A$. For $T_j>K$,
\[
 \frac{\sup_{t_0\leq t\leq K}F_{T_j}(t)}{F_{T_j}(T_j)}
 \leq\frac{M_K e^{-\delta(T_j-K)}}{e^{-\delta T_j/2}}
 =M_K e^{\delta K-\delta T_j/2}\longrightarrow0.
\]
Hence $s_j>K$ for all sufficiently large $j$, so that
\[
 s_j\longrightarrow\infty,\qquad a_j=A(s_j)\longrightarrow0,
\]
and \eqref{t134:ni:eq:envelope} holds in this case as well.
\end{proof}

\begin{lemma}
\label{t134:ni:lem:harmonic-vanishing}
Let $\kappa>0$ be noninteger, and let $\delta$ satisfy
\begin{equation}\label{t134:ni:eq:delta}
 0<\delta<\min\{\kappa-\lfloor\kappa\rfloor,
                         \lfloor\kappa\rfloor+1-\kappa\}.
\end{equation}
Let $V$ be a real harmonic function in $\HH$, with
$V\in H^1(B_R^+)$ for every $R>0$. Assume that $V$ has either
zero Neumann trace or zero Dirichlet trace on
$\partial\HH$, and that
\begin{equation}\label{t134:ni:eq:window}
 \int_{B_R^+}V^2\leq C R^{2\kappa+2}e^{2\delta|\log R|}
 \quad\text{for every }R>0,
\end{equation}
with $C>0$ independent of $R$. Then $V\equiv0$.
\end{lemma}

\begin{proof}
Even reflection in the Neumann case and odd reflection in the Dirichlet
case extend $V$ to an entire harmonic function $\widetilde V$.
Set $m=\lfloor\kappa\rfloor$; then
$m<\kappa-\delta<\kappa+\delta<m+1$.
By harmonic expansion and angular orthogonality, the bound
\eqref{t134:ni:eq:window} as $R\to\infty$ implies that $\widetilde V$
is a harmonic polynomial of degree at most $m$.
The same bound as $R\rightarrow0$, together with $\kappa-\delta>m$,
forces all its homogeneous terms to vanish. Hence $V\equiv0$.
\end{proof}

\begin{proposition}\label{t134:ni:prop:upper}
Let $\Omega\subset\R^2$ be a bounded connected $C^1$ domain with
$0\in\partial\Omega$, in the normalized graph coordinates
\eqref{t134:eq:graph}. Let $0<p<1$, $\lambda\geq0$, and
$\kappa=(1-p)^{-1}$. Suppose that $u$ is real-valued and satisfies
the local problem \eqref{paper:eq:local-problem}, with the
regularity \eqref{paper:eq:local-class}. Assume that
\[
 u(0)=0,\qquad \ord_0u=\kappa\notin\N.
\]

Then there exist constants $C>0$ and $r_*>0$ such that
\[
 \int_{\Omega\cap B_r}u^2\leq Cr^{2\kappa+2}
 \qquad\text{for every }0<r<r_*.
\]
\end{proposition}

\begin{proof}
Fix $\delta$ satisfying \eqref{t134:ni:eq:delta}.
If $A$ were unbounded, \eqref{t134:ni:eq:order-facts} and
Lemma~\ref{t134:ni:lem:scale-selection}\textup{(i)} would give
$s_j\to\infty$ and $a_j=A(s_j)\to\infty$ satisfying
\eqref{t134:ni:eq:envelope}. Set $r_j=e^{-s_j}$ and define on $\Omega_{r_j}$
\[
 v_j(X)=\frac{u(r_jX)}{r_j^\kappa a_j}.
\]
The boundary parameters are
\[
 \epsilon_j=\lambda r_j\to0,\qquad
 \mu_j=a_j^{p-1}\to0,
\]
because $1+\kappa(p-1)=0$. For every fixed $R>0$ and all sufficiently
large $j$,
\[
 \int_{\Omega_{r_j}\cap B_1}v_j^2=1,
 \qquad
 \int_{\Omega_{r_j}\cap B_R}v_j^2
 =R^{2\kappa+2}\frac{A(s_j-\log R)^2}{a_j^2}
 \leq R^{2\kappa+2}e^{2\delta|\log R|}.
\]
Lemma~\ref{t134:ni:lem:compactness} gives a global Neumann harmonic limit
$V$ with $\int_{B_1^+}V^2=1$ and \eqref{t134:ni:eq:window}.
Lemma~\ref{t134:ni:lem:harmonic-vanishing} then yields $V\equiv0$,
contradicting the normalization.
\end{proof}

For a flat global solution of
\begin{equation}\label{t134:ni:eq:flat-limit}
 \Delta V=0\quad\text{in }\HH,
 \qquad -V_y=-|V|^{p-1}V\quad\text{on }\R,
\end{equation}
let
\begin{equation*}
 W(V,R)=R^{-2\kappa}\left(
 \int_{B_R^+}|\nabla V|^2+
 \frac2{p+1}\int_{-R}^R|V(x,0)|^{p+1}\dd x\right)
 -\kappa R^{-2\kappa-1}\int_{S_R^+}V^2.
\end{equation*}

It is easy to see that the $\kappa$-homogeneous solutions of \eqref{t134:ni:eq:flat-limit}
are precisely $0,P_+,P_-$,
where, as in Theorem~\ref{t134:ni:thm:main},
\[
P_\pm(R,\theta)=\pm\frac{a_\kappa}{\cos\beta}R^\kappa
 \cos(\kappa\theta-\beta),\qquad R>0,\quad0\leq\theta\leq\pi,
\]
with
\[
 \beta=\frac\pi2\bigl(\kappa-\lfloor\kappa\rfloor\bigr)
 \in(0,\pi/2),\qquad
 a_\kappa=(\kappa\tan\beta)^{-\kappa}.
\]
The Weiss values of $0$, $P_{+}$ and $P_{-}$ are
\begin{equation}\label{t134:ni:eq:critical-values}
 \mathcal E=\{0,W_*\},\qquad
 W_*:=W(P_\pm,1)=\frac{a_\kappa^{p+1}}{\kappa^2(p+1)}>0.
\end{equation}

In the original coordinates, define
\begin{equation*}
 \begin{split}
 \mathcal W(u,r)={}r^{-2\kappa}\left(
 \int_{\Omega\cap B_r}|\nabla u|^2
 +\frac2{p+1}\int_{\partial\Omega\cap B_r}|u|^{p+1}\right)
 -\kappa r^{-2\kappa-1}
             \int_{\Omega\cap\partial B_r}u^2,
 \end{split}
\end{equation*} and
 $F(t)=\mathcal W(u,e^{-t})$.

\begin{lemma}\label{t134:ni:lem:weiss-compactness}
The function $F$ is bounded and continuous for all sufficiently large
$t$.  For every sequence $t_j\to\infty$ there is a subsequence and a
global solution $V$ of \eqref{t134:ni:eq:flat-limit} such that, whenever
$s_j\to s\in\R$,
\begin{equation}\label{t134:ni:eq:translation-limit}
 F(t_j+s_j)\longrightarrow W(V,e^{-s}).
\end{equation}
Moreover,
\[
 \lim_{t\to\infty}F(t)=L\in\mathcal E.
\]
\end{lemma}

\begin{proof}
Proposition~\ref{t134:ni:prop:upper} and the local energy and
Moser estimates give
\[
 \int_{\Omega\cap B_r}|\nabla u|^2
 +\int_{\partial\Omega\cap B_r}|u|^{p+1}\le Cr^{2\kappa},
 \qquad\sup_{\Omega\cap B_r}|u|\le Cr^\kappa.
\]
Since the circular arc has length at most $2\pi r$, $F$ is bounded.
The radial function on either boundary branch is strictly monotone
near zero: its derivative has numerator $x+\psi\psi'$, of the
sign of $x$. Thus the integration endpoints vary continuously,
and $F$ is continuous.

Set $r_j=e^{-t_j}$ and $U_j=r_j^{-\kappa}u(r_j\,\cdot)$.
Lemma~\ref{t134:ni:lem:compactness} with $\mu_j=1$ gives
$U_j\circ T_{r_j}\to V$ strongly in local $H^1$ and locally
uniformly up to the diameter. By scaling,
\[
 F(t_j+s_j)=\mathcal W_{\Omega_{r_j}}(U_j,e^{-s_j}).
\]
For $e^{-s_j}\to R>0$, strong gradient convergence and convergence
of the pulled-back ball indicators give convergence of the bulk
energy. Uniform trace convergence, graph length densities tending
to one, and endpoints tending to $\pm R$ treat the boundary term.
On the circular arc, parameterization by
$e^{-s_j}(\cos\theta,\sin\theta)$ and uniform convergence give
the limit $\int_{S_R^+}V^2$; the endpoint discrepancies vanish.

It remains to prove convergence of $F$.
By \eqref{t134:ni:eq:translation-limit} and the flat Weiss monotonicity
formula,
\begin{equation}\label{t134:ni:eq:one-sided}
 \sup_{0\le s\le1}(F(t+s)-F(t))_+\longrightarrow0.
\end{equation}
Indeed, a violating sequence with $s_j\to s$ would contradict
$W(V,e^{-s})\le W(V,1)$. Similarly, for every compact interval
$[a,b]$ disjoint from $\mathcal E$, some $\delta_0>0$ and $T$
satisfy
\begin{equation}\label{t134:ni:eq:strict-drop}
 F(t+1)\le F(t)-\delta_0
 \quad(t\ge T,\ F(t)\in[a,b]).
\end{equation}
Otherwise, a limiting solution $V$ would satisfy
\[
 W(V,e^{-1})=W(V,1),\qquad W(V,1)\in[a,b].
\]
By Lemma~\ref{t134:lem:flat-weiss}, $V$ is $\kappa$-homogeneous;
hence \eqref{t134:ni:eq:critical-values} gives $W(V,1)\in\mathcal E$,
contrary to $[a,b]\cap\mathcal E=\varnothing$.

If $\liminf F(n)<\limsup F(n)$, choose $[a,b]$ strictly inside
this gap and disjoint from the finite set $\mathcal E$.
For any $0<\eta<b-a$, all sufficiently late positive increments
are at most $\eta$. Once $F(n)\le a$, induction using
\eqref{t134:ni:eq:strict-drop} confines subsequent values to
$(-\infty,a+\eta]$, a contradiction. Hence $F(n)\to L$.
A strict drop near $L$ would contradict convergence, so
$L\in\mathcal E$. For $n\le t\le n+1$,
\eqref{t134:ni:eq:one-sided} gives
$F(n+1)-o(1)\le F(t)\le F(n)+o(1)$ uniformly, proving the claim.
\end{proof}

\begin{proposition}\label{t134:ni:prop:nondegenerate}
Let
\[
 L:=\lim_{t\to\infty}F(t)
   =\lim_{r\rightarrow0}\mathcal W(u,r),
\]
whose existence follows from Lemma~\ref{t134:ni:lem:weiss-compactness}.
Then $L=W_*>0$, where $W_*$ is defined in
\eqref{t134:ni:eq:critical-values}.
In particular, there exist constants $0<c\leq C<\infty$ and
$r_*\in(0,R_0)$ such that
\[
 c r^{2\kappa+2}\leq\int_{\Omega\cap B_r}u^2
 \leq C r^{2\kappa+2}
 \qquad\text{for every }0<r<r_*.
\]
\end{proposition}

\begin{proof}
Every sequence of natural blow-ups has a locally uniform and strong
local $H^1$ cluster limit $V$ by Lemma~\ref{t134:ni:lem:compactness}.
Lemma~\ref{t134:ni:lem:weiss-compactness} gives
\[
 W(V,R)=L\qquad\text{for every }R>0.
\]
The flat Weiss identity forces $V$ to be $\kappa$-homogeneous.
Thus every cluster limit lies in $\{0,P_+,P_-\}$ and has Weiss
value $L$.

Suppose $L=0$.  The only homogeneous cluster limit with zero Weiss
value is zero.  Precompactness therefore gives convergence of the
full natural family to zero, including its $L^2$ norm on 
unit half-balls.  Hence $A(t)\to0$.

Fix $\delta$ satisfying \eqref{t134:ni:eq:delta}.
By \eqref{t134:ni:eq:order-facts} and
Lemma~\ref{t134:ni:lem:scale-selection}\textup{(ii)}, there exist
$s_j\to\infty$ and $a_j=A(s_j)\to0$ satisfying
\eqref{t134:ni:eq:envelope}. Set
\[
 r_j=e^{-s_j},\qquad
 v_j(X)=\frac{u(r_jX)}{r_j^\kappa a_j}.
\]
The same normalization and growth envelope as in the proof of
Proposition \ref{t134:ni:prop:upper} hold.  This time
$\mu_j=a_j^{p-1}\to\infty$.  Lemma \ref{t134:ni:lem:compactness} gives a
global zero-trace harmonic limit $V$ with
$\int_{B_1^+}V^2=1$ and the degree-window bound \eqref{t134:ni:eq:window}.
Lemma~\ref{t134:ni:lem:harmonic-vanishing} gives $V\equiv0$,
contradicting the normalization.

Thus $L=W_*>0$, and every cluster limit is $P_+$ or $P_-$.
If the asserted lower bound failed, a sequence with $\mathcal M(r_j)\to0$
would have a zero cluster limit, again a contradiction.
\end{proof}

\begin{proof}[\normalfont\bfseries Proof of Theorem \ref{t134:ni:thm:main}]

Consider the continuous function
\[
 b(r)=r^{-\kappa}u(r,\psi(r)),\qquad r>0.
\]
It has no zeros at arbitrarily small radii: a sequence of such zeros
would, by compactness, yield a cluster profile whose value at $(1,0)$
is zero, whereas $P_\pm(1,0)=\pm a_\kappa\neq0$.  Hence $b(r)$ has a
constant sign for all sufficiently small $r$.  That sign selects
exactly one of $P_+,P_-$, so every sequence has the same cluster
limit.  This proves the full convergence \eqref{t134:ni:eq:convergence}.

The volume convergence in Lemma \ref{t134:ni:lem:compactness} gives
\[
 \lim_{r\rightarrow0}r^{-2\kappa-2}
 \int_{\Omega\cap B_r}u^2=\int_{B_1^+}P_\sigma^2.
\]
Since
\[
 \int_0^\pi\cos^2(\kappa\theta-\beta)\dd\theta
 =\frac\pi2+\frac{\sin(2\beta)}{2\kappa},
\]
radial integration gives exactly \eqref{t134:ni:eq:density-limit}.

Finally $\psi(x)=o(|x|)$ uniformly near zero.  The inverse graph map
$T_r^{-1}$ tends uniformly to the identity on bounded sets.  The
homogeneous profiles extend continuously to any fixed slightly
larger sector, and the  domain near zero lies in such a
sector.  For $X\in\overline\Omega$ with $r=|X|$, apply
\eqref{t134:ni:eq:convergence} at $T_r^{-1}(X/r)$, which remains in a fixed
closed half-disk.  Uniform continuity of $P_\sigma$ on the relevant
bounded sector gives
\[
 r^{-\kappa}u(X)-P_\sigma(X/r)\longrightarrow0
\]
uniformly over all these $X$.  Homogeneity proves
\eqref{t134:ni:eq:physical}.
\end{proof}

\subsection{Integer critical order with a continuous positive weight}
\label{t134:int:sec:integer}

We exclude integer critical order for the flat problem with a
continuous positive boundary weight. The conformal reduction to
Theorem~\ref{thm:integer-critical-order} is given after the proof.

\begin{theorem}\label{t134:int:thm:weighted}
Let $N\ge2$ be an integer and set $p=1-N^{-1}$.
Let $R_0>0$, $\lambda\ge0$, and $\rho\in C(I_{R_0})$ satisfy
$0<c_0\le\rho\le c_1$ on $I_{R_0}$.
Suppose that $U\not\equiv0$ is real-valued and
\[
 U\in H^1(B_r^+)\cap C(\overline{B_r^+})
 \qquad\text{for every }0<r<R_0.
\]
Assume that $U$ satisfies
\[
 \begin{cases}
  \Delta U=0 &\text{in }B_{R_0}^+,\\
  -U_y=\rho(x)\bigl(\lambda U-|U|^{p-1}U\bigr)
   &\text{on }I_{R_0}.
 \end{cases}
\]
If $U(0)=0$, then
\[
 \ord_0U\ne N.
\]
\end{theorem}

\begin{proof}
Suppose, for a contradiction, that $\ord_0U=N$, and put
$\rho_0=\rho(0)$. We first show that the critical density is
bounded and tends to zero, and then use the boundary trace of a
Dirichlet blow-up to obtain a contradiction.

\par\medskip
\Needspace{4\baselineskip}
\phantomsection\label{t134:int:step:setup}
\noindent\textbf{Step 1: Critical density and scalar projections.}

Choose
$r_0<R_0/8$, set $t_0=-\log r_0$, and define
\[
A(t)=\left(e^{(2N+2)t}\int_{B_{e^{-t}}^+}U^2\right)^{1/2},
\qquad t\ge t_0.
\]
Then $A$ is continuous and strictly positive, and
\begin{align}
A(t)&\le C_\varepsilon e^{\varepsilon t}
&&\text{for every }\varepsilon>0,\label{t134:int:eq:subexp}\\
A(t)&\ne O(e^{-\varepsilon t})
&&\text{for every }\varepsilon>0.\label{t134:int:eq:notexp}
\end{align}
Write
\[
v(t,\theta)=e^{Nt}U(e^{-t+i\theta}),\qquad0\le\theta\le\pi,
\]
and let $v_+(t)=v(t,0)$, $v_-(t)=v(t,\pi)$,
$\rho_\pm(t)=\rho(\pm e^{-t})$.
The boundary Moser estimate, with a harmless change of constants,
gives
\begin{equation}\label{t134:int:eq:moser}
\|v(t,\cdot)\|_\infty\le C A(t-h),\qquad h=\log2,
\quad t\ge t_0+h.
\end{equation}
Define the real scalar projections
\[
a(t)=\frac2\pi\int_0^\pi v(t,\theta)\cos(N\theta)\dd\theta,
\qquad
b(t)=\frac2\pi\int_0^\pi v(t,\theta)\sin(N\theta)\dd\theta.
\]
Both are subexponentially bounded in the sense of
\eqref{t134:int:eq:subexp}. They belong to $C^2$ for $t>t_0$: initially the
identities below hold distributionally by Green's formula,
and their right-hand sides are continuous.
\begin{align*}
a''-2Na'&=S(t),\\
S(t)&=\frac2\pi\left[
\rho_+|v_+|^{p-1}v_++(-1)^N\rho_-|v_-|^{p-1}v_-
-\lambda e^{-t}\bigl(\rho_+v_++(-1)^N\rho_-v_-\bigr)
\right],\\
b''-2Nb'&=-\frac{2N}{\pi}
\bigl(v_+-(-1)^Nv_-\bigr).
\end{align*}
For example, the cylinder equation is
$v_{tt}-2Nv_t+v_{\theta\theta}+N^2v=0$. At its two ends,
\[
v_\theta(t,0)=\rho_+|v_+|^{p-1}v_+-\lambda e^{-t}\rho_+v_+,
\quad
v_\theta(t,\pi)=-\rho_-|v_-|^{p-1}v_-+\lambda e^{-t}\rho_-v_-.
\]
Integration against the sine or cosine gives the formulas above.

Subexponential growth removes the unstable homogeneous solution of
each scalar ODE. Consequently,
\begin{align}
a'(t)&=-\int_0^\infty e^{-2N\tau}S(t+\tau)\dd\tau,
\label{t134:int:eq:aint}\\
b'(t)&=\frac{2N}{\pi}\int_0^\infty e^{-2N\tau}
\bigl(v_+(t+\tau)-(-1)^Nv_-(t+\tau)\bigr)\dd\tau.
\label{t134:int:eq:bint}
\end{align}
Indeed, the forcing is subexponentially bounded by
\eqref{t134:int:eq:moser}; solving the first-order equation for $a'$ or $b'$
shows that a nonzero omitted constant would produce an
$e^{2Nt}$ term in the projection itself.

For a subexponentially bounded scalar function $c$, put
\[
\mathcal E_\delta[c](t)=
\sup_{s\ge t_0}|c(s)|e^{-\delta|s-t|},\qquad\delta>0.
\]
It is finite and obeys
\begin{equation}\label{t134:int:eq:elip}
\mathcal E_\delta[c](s)\le
e^{\delta|s-t|}\mathcal E_\delta[c](t).
\end{equation}
Similarly $A(s)e^{-\delta|s-t|}$ attains a positive maximum at some
$s=s(t)$. At a maximizer,
\begin{equation*}
A(\tau)\le A(s)e^{\delta|\tau-s|}
\qquad(\tau\ge t_0).
\end{equation*}

\par\medskip
\Needspace{4\baselineskip}
\phantomsection\label{t134:int:step:compactness}
\noindent\textbf{Step 2: Compactness and classification of rescaled limits.}

We use the following general rescaling template, specifying the scales
and normalizing amplitudes separately in each application.
For scales $r_j\to0$ and amplitudes $d_j>0$, set
\[
 w_j(z)=\frac{U(r_jz)}{r_j^Nd_j},\qquad
 \rho_j(x)=\rho(r_jx),\qquad
 \epsilon_j=\lambda r_j,\qquad \mu_j=d_j^{p-1}.
\]
The density controls their masses through
\[
 \int_{B_R^+}w_j^2
 =R^{2N+2}\frac{A(-\log r_j-\log R)^2}{d_j^2}
 \qquad\text{for each fixed $R>0$ and all large $j$}.
\]
Criticality gives
\[
 \Delta w_j=0,\qquad
 -(w_j)_y=\rho_j(x)
       \bigl(\epsilon_jw_j-\mu_j|w_j|^{p-1}w_j\bigr).
\]
The weights converge locally uniformly to $\rho_0=\rho(0)>0$.
Whenever the local $L^2$ masses are uniformly bounded,
Lemma~\ref{t134:ni:lem:compactness}, with $\psi_j\equiv0$,
gives uniform local $H^1$ and supremum bounds and strong local
$H^1$ compactness. If $\mu_j\to\mu_0<\infty$, the convergence
is locally uniform up to the boundary and the limit satisfies
$-W_y=-\mu_0\rho_0|W|^{p-1}W$. If $\mu_j\to\infty$, the limit
has zero trace and the boundary penalty tends to zero.
The same compactness statement applies to any sequence satisfying
these flat equations and bounds, including the constant-weight
rescalings used below. Strong $H^1$ convergence on a larger
half-disk also gives strong $L^2(S_1^+)$ convergence by the trace
theorem on a fixed half-annulus.

For later use, suppose an entire half-plane limit $W$ has either
zero Neumann data or zero trace and satisfies, for $0<\delta<1$,
\[
 \int_{B_R^+}W^2\le C R^{2N+2}e^{2\delta|\log R|}
 \qquad(R>0).
\]
Even or odd reflection gives an entire harmonic function. Its
orthogonal harmonic power expansion and the bounds at zero and
infinity permit only integer degrees in $[N-\delta,N+\delta]$.
Thus $W=c\operatorname{Re}z^N$ in the Neumann case and
$W=c\operatorname{Im}z^N$ in the Dirichlet case.

\par\medskip
\Needspace{4\baselineskip}
\phantomsection\label{t134:int:step:boundedness}
\noindent\textbf{Step 3: The critical density is bounded.}

Fix $0<\delta<1$. We first prove
\[
 A(t)\le C_\delta\bigl(1+\mathcal E_\delta[a](t)\bigr)
 \qquad(t\ge t_0).
\]
If this failed, the ratio would be unbounded along a sequence $\widehat t_j$.
Maximize
$A(s)e^{-\delta|s-\widehat t_j|}$ at $\xi_j$, and write $m_j=A(\xi_j)$.
Using \eqref{t134:int:eq:elip} and maximality gives
\[
\frac{m_j}{1+\mathcal E_\delta[a](\xi_j)}
\ge
\frac{A(\widehat t_j)}{1+\mathcal E_\delta[a](\widehat t_j)}\longrightarrow\infty.
\]
Thus $m_j\to\infty$, $\xi_j\to\infty$, and the density-normalized functions
\[
u_j(z)=\frac{U(e^{-\xi_j}z)}{e^{-N\xi_j}m_j}
\]
have unit $L^2(B_1^+)$ norm and the bilateral growth bound in
\hyperref[t134:int:step:compactness]{Step~2}. Their parameters satisfy
$\epsilon_j=\lambda e^{-\xi_j}\to0$ and $\mu_j=m_j^{p-1}\to0$.
The compactness and classification in \hyperref[t134:int:step:compactness]{Step~2} give a nonzero limit
$W=c\operatorname{Re}z^N$. But
$a(\xi_j)/m_j\to0$, whereas convergence of the trace on $S_1^+$ gives
$a(\xi_j)/m_j\to c\ne0$. This is a contradiction.

We now show that $A$ is bounded. Otherwise the preceding comparison
would force $|a|$ to be unbounded. Choose $M_j\uparrow\infty$,
$M_j>|a(t_0)|$. Let $T_j$ be the first time that $|a|=M_j$.
Maximize
\[
|a(s)|e^{-\delta(s-T_j)}\qquad(s\ge T_j)
\]
at $t_j$. The maximum exists by subexponential growth. Put
$c_j=|a(t_j)|\ge M_j$ and, after extraction,
$\varsigma=\sgn a(t_j)\in\{-1,1\}$. At an interior maximizer
$\varsigma a'(t_j)=\delta c_j>0$; if $t_j=T_j$, the first-hitting
property gives $\varsigma a'(t_j)\ge0$. Thus
\begin{equation}\label{t134:int:eq:apos}
\varsigma a'(t_j)\ge0.
\end{equation}
For every $s\le t_j$, one has $|a(s)|\le c_j$: before $T_j$ this
uses first hitting, and between $T_j$ and $t_j$ it uses maximality.
For $s\ge t_j$, maximality gives
$|a(s)|\le c_j e^{\delta(s-t_j)}$. Hence
\[
\mathcal E_\delta[a](s)\le c_j e^{\delta(s-t_j)_+},
\qquad
A(s)\le C c_j e^{\delta(s-t_j)_+}
\]
for large $j$, with $C$ independent of $j,s$.

Normalize by the cosine amplitude $c_j$, with $r_j=e^{-t_j}$:
$f_j(z)=U(r_jz)/(r_j^Nc_j)$. The resulting growth window has only
the degree $N$. By \hyperref[t134:int:step:compactness]{Step~2}, after extraction,
\[
f_j\longrightarrow\varsigma\operatorname{Re}z^N
\]
strongly in local $H^1$ and locally uniformly on closed half-disks.
The coefficient is exactly $\varsigma$ because the cosine projection
on $S_1^+$ equals $a(t_j)/c_j=\varsigma$.

For each fixed $\tau\ge0$, this gives
\[
\frac{v_+(t_j+\tau)}{c_j}\to\varsigma,
\qquad
\frac{v_-(t_j+\tau)}{c_j}\to(-1)^N\varsigma.
\]
Using \eqref{t134:int:eq:moser} and the scalar envelope, the nonlinear part
of $c_j^{-p}S(t_j+\tau)$ is bounded by $Ce^{p\delta\tau}$.
Its linear part is bounded by
$C e^{-t_j}c_j^{1-p}e^{(\delta-1)\tau}$.
The prefactor tends to zero because $a$ is subexponentially bounded.
Dominated convergence in \eqref{t134:int:eq:aint} therefore yields
\[
\frac{a'(t_j)}{c_j^p}
\longrightarrow
-\frac{4\rho_0\varsigma}{\pi}\int_0^\infty e^{-2N\tau}\dd\tau
=-\frac{2\rho_0\varsigma}{\pi N}.
\]
This contradicts \eqref{t134:int:eq:apos}, so $A$ is bounded.

\par\medskip
\Needspace{4\baselineskip}
\phantomsection\label{t134:int:step:density-zero}
\noindent\textbf{Step 4: The critical density tends to zero.}

By \hyperref[t134:int:step:boundedness]{Step~3}, the natural rescalings
$r^{-N}U(r\cdot)$ have volume bounds $CR^{2N+2}$ for every fixed
$R>0$, once $r$ is sufficiently small. Their parameters are
$\epsilon=\lambda r\to0$ and $\mu=1$. By \hyperref[t134:int:step:compactness]{Step~2},
every sequence $r\to0$ has a strong local $H^1$ subsequential limit
$Z$ satisfying
\[
 \Delta Z=0\quad\hbox{in }\R^2_+,\qquad
 -Z_y=-\rho_0|Z|^{p-1}Z\quad\hbox{on }\R,
 \qquad \int_{B_R^+}Z^2\le CR^{2N+2}\quad(R>0).
\]
We show directly that $Z=0$.

The constant-coefficient Weiss functional
\[
\mathcal W(Z,R)=R^{-2N}\left(\int_{B_R^+}|\nabla Z|^2+
\frac{2\rho_0}{p+1}\int_{I_R}|Z|^{p+1}\right)
-NR^{-2N-1}\int_{S_R^+}Z^2
\]
satisfies, by Lemma~\ref{t134:lem:flat-weiss},
\[
\frac{d}{dR}\mathcal W(Z,R)=
2R^{-2N}\int_{S_R^+}\left(\partial_rZ-\frac NRZ\right)^2\ge0.
\]
The stated volume bound, the cutoff energy estimate, and the trace
inequality show that $\mathcal W(Z,R)$ is bounded uniformly for all $R>0$.
Thus it has finite limits at zero and infinity.

Take any sequence $R_j\to0$, or $R_j\to\infty$, and rescale
$Z_j(z)=R_j^{-N}Z(R_jz)$. Criticality preserves the equation.
\hyperref[t134:int:step:compactness]{Step~2} gives local strong $H^1$ compactness. For any
$0<R_-<R_+$, the difference
$\mathcal W(Z,R_jR_+)-\mathcal W(Z,R_jR_-)$ tends to zero. The limit is therefore
$N$-homogeneous.

There is no nonzero homogeneous solution of this integer degree.
Indeed, write such a homogeneous limit as
$P=r^N(\alpha\cos N\theta+\beta\sin N\theta)$.
The two boundary equations imply
\[
N\beta=\rho_0|\alpha|^{p-1}\alpha,\qquad
(-1)^NN\beta=-(-1)^N\rho_0|\alpha|^{p-1}\alpha,
\]
so $\alpha=\beta=0$. It follows that the rescaled limits at both ends are
zero. Strong $H^1$ convergence, including the surface and flat trace
terms, gives
$\lim_{R\rightarrow0}\mathcal W(Z,R)=\lim_{R\uparrow\infty}\mathcal W(Z,R)=0$.
Monotonicity then forces $\mathcal W\equiv0$, and $Z$ is homogeneous and zero.

Every subsequential limit is therefore zero. The whole family
$r^{-N}U(r\cdot)$ converges to zero in local $H^1$, and its
$L^2(B_1^+)$ norm $A(-\log r)$ tends to zero.

\par\medskip
\Needspace{4\baselineskip}
\phantomsection\label{t134:int:step:sine-envelope}
\noindent\textbf{Step 5: Control by the sine projection.}

Fix $0<\delta<1$. We next prove that $b\not\equiv0$ and
\[
 A(t)\le C_\delta\mathcal E_\delta[b](t)\qquad(t\ge t_0),
\]
and that $b$ does not decay exponentially. Since $A\to0$ by
\hyperref[t134:int:step:density-zero]{Step~4} and
\eqref{t134:int:eq:notexp} holds, there are $T_j\to\infty$ with
$A(T_j)\ge e^{-\delta T_j/2}$. Maximize
$A(s)e^{-\delta|s-T_j|}$ at $\xi_j$. On every fixed compact
$[t_0,M]$ its maximum is $O(e^{-\delta(T_j-M)})=o(A(T_j))$.
Hence $\xi_j\to\infty$ and $m_j=A(\xi_j)\to0$. This gives a new
density-normalized sequence
$u_j(z)=U(e^{-\xi_j}z)/(e^{-N\xi_j}m_j)$ with unit
$L^2(B_1^+)$ norm and
the bilateral growth bound of \hyperref[t134:int:step:compactness]{Step~2}.
Since $\mu_j=m_j^{p-1}\to\infty$, that step gives a nonzero
limit $c\operatorname{Im}z^N$, proving
$b(\xi_j)/m_j\to c\ne0$.

If the displayed comparison failed, choose $\widehat t_j$ such that
$A(\widehat t_j)/\mathcal E_\delta[b](\widehat t_j)\to\infty$ and maximize
$A(s)e^{-\delta|s-\widehat t_j|}$ at $\xi_j$, as in
\hyperref[t134:int:step:boundedness]{Step~3}. For this new sequence, again put
$m_j=A(\xi_j)$ and $u_j(z)=U(e^{-\xi_j}z)/(e^{-N\xi_j}m_j)$. The ratio
$A(\xi_j)/\mathcal E_\delta[b](\xi_j)$ is at least its value at $\widehat t_j$.
Since $b\not\equiv0$, its envelope has a strictly positive lower
bound on every compact interval. Consequently $\xi_j\to\infty$,
$m_j\to0$, and the nonzero Dirichlet limit again contradicts
$b(\xi_j)/m_j\to0$.

Finally, if $|b(t)|\le Ce^{-\varepsilon t}$, choose
$0<\delta<\min(1,\varepsilon)$. Direct maximization gives
$\mathcal E_\delta[b](t)\le C_\delta e^{-\delta t}$, including the
fixed initial interval. The comparison just proved would imply exponential decay of
$A$, contrary to \eqref{t134:int:eq:notexp}.

\par\medskip
\Needspace{4\baselineskip}
\phantomsection\label{t134:int:step:scales}
\noindent\textbf{Step 6: Selection of a Dirichlet blow-up.}

By \hyperref[t134:int:step:density-zero]{Step~4} and \eqref{t134:int:eq:moser}, $b\to0$.
Fix $0<\delta<1$. The nonexponential decay established in
\hyperref[t134:int:step:sine-envelope]{Step~5} gives the following choice of scales.
There are $R_j\to\infty$ such that
$|b(R_j)|\ge e^{-\delta R_j/2}$. Put $M_j=|b(R_j)|>0$, and let
$T_j$ be the last time that $|b|=M_j$. It exists since $b\to0$ and
satisfies $T_j\ge R_j$. Maximize
\[
|b(s)|e^{-\delta(T_j-s)}\qquad(t_0\le s\le T_j)
\]
at $s_j$, and put $d_j=|b(s_j)|$.

The maximizing points escape to infinity. Indeed, on a fixed
$[t_0,M]$ the maximum is at most $C_Me^{-\delta(T_j-M)}$, whereas
the global maximum is at least
$M_j\ge e^{-\delta R_j/2}\ge e^{-\delta T_j/2}$.
Thus $s_j\to\infty$, $d_j\to0$, and, after extraction,
$\sigma=\sgn b(s_j)$ is fixed. At an interior maximizer,
$\sigma b'(s_j)=-\delta d_j<0$; if $s_j=T_j$, the last-crossing
property gives $\sigma b'(s_j)\le0$. Hence
\begin{equation}\label{t134:int:eq:bnegative}
\sigma b'(s_j)\le0.
\end{equation}

Maximality gives $|b(s)|\le d_je^{\delta(s_j-s)}$ for $s\le s_j$,
and $|b(s)|\le d_j$ for $s_j\le s\le T_j$. For $s\ge T_j$, the
last-crossing property gives $|b(s)|\le M_j\le d_j$. Therefore
\[
\mathcal E_\delta[b](s)\le d_j e^{\delta(s_j-s)_+},
\qquad A(s)\le C d_j e^{\delta(s_j-s)_+}.
\]
Normalize with $r_j=e^{-s_j}$:
\[
w_j(z)=\frac{U(r_jz)}{r_j^Nd_j},\qquad
\mu_j=d_j^{p-1}\to\infty,\qquad\epsilon_j=\lambda r_j\to0.
\]
After extraction, the growth bound,
\hyperref[t134:int:step:compactness]{Step~2}, and the sine projection give
\begin{equation}\label{t134:int:eq:dirlim}
w_j\longrightarrow\sigma\operatorname{Im}z^N
\quad\hbox{strongly in local }H^1.
\end{equation}
The normalization is again exact, since the sine projection on
$S_1^+$ equals $b(s_j)/d_j=\sigma$. We retain this sequence and
the notation $s_j,d_j,r_j,w_j$ throughout the remaining steps.

\par\medskip
\Needspace{4\baselineskip}
\phantomsection\label{t134:int:step:boundary-trace}
\noindent\textbf{Step 7: Recovery of the normalized boundary trace.}

The zero-trace limit obtained in \hyperref[t134:int:step:scales]{Step~6}
does not by itself yield a contradiction. We must identify the boundary
trace at a finer scale to determine the sign of $b'(s_j)$ through
\eqref{t134:int:eq:bint}.

Keep the rescalings chosen in \hyperref[t134:int:step:scales]{Step~6} and put
$V=\sigma\operatorname{Im}z^N$, $\rho_j(x)=\rho(r_jx)$,
and $\kappa_j=\mu_j^{-1/p}$. On a fixed half-disk $B_2^+$,
\begin{equation}\label{t134:int:eq:original}
 \partial_\nu w_j+\rho_j(x)
       \bigl(\mu_j |w_j|^{p-1}w_j-\epsilon_jw_j\bigr)=0.
\end{equation}
Here $\mu_j\to\infty$, $\epsilon_j\to0$,
$0<\rho_*\le\rho_j\le\rho^*$, and $\rho_j\to\rho_0>0$
locally uniformly. By \hyperref[t134:int:step:compactness]{Step~2} and
\eqref{t134:int:eq:dirlim}, the $w_j$ are uniformly locally bounded
up to the diameter and converge locally in $L^2$ to $V$,
which has zero trace and odd harmonic reflection.
We will prove
\begin{equation}\label{t134:int:eq:scaledbound}
 \sup_{x\in K}|w_j(x,0)|\le C_K\kappa_j
 \qquad(K\Subset I_2),
\end{equation}
and, for every compact interval $K\Subset I_2$ and $1\le\ell<\infty$,
\begin{equation}\label{t134:int:eq:strongtrace}
 z_j:=\kappa_j^{-1}w_j(\cdot,0)
 \longrightarrow
 g:=\sgn\bigl(V_y(\cdot,0)\bigr)
       \left(\frac{|V_y(\cdot,0)|}{\rho_0}\right)^{1/p}
 \quad\hbox{strongly in }L^\ell(K).
\end{equation}
The argument includes zeros of $V_y$.

\emph{(a) A uniform bound at the penalization scale.}
Fix a compact interval $K\Subset I_2$.  Choose $d>0$ so small that,
for every $x_0\in K$, the closure of $B_d^+(x_0,0)$ is contained in
a fixed relatively compact half-disk.  Let $M\ge1$ be a common
bound for $|w_j|$ on these half-disks, and set
\[
 \Psi_j(x,y)=M\left(\frac{2y}{d}
             +\frac{(x-x_0)^2-y^2}{d^2}\right)+c\kappa_j.
\]
This function is harmonic.  On the semicircle $S_d^+(x_0,0)$,
putting $s=y/d\in[0,1]$, its parenthesis equals
$1+2s-2s^2\ge1$, so $\Psi_j\ge M\ge|w_j|$.  On the flat diameter,
\[
 c\kappa_j\le \Psi_j\le M+c\kappa_j,
 \qquad \partial_\nu \Psi_j=-\frac{2M}{d}.
\]
Choose a fixed $c$ so large that
\[
 \rho_*c^p\ge\frac{2M}{d}
       +\rho^*\sup_j\epsilon_j\,(M+1)+1.
\]
For sufficiently large $j$, $c\kappa_j\le1$, and hence on the
diameter
\[
 \partial_\nu \Psi_j+
 \rho_j\bigl(\mu_j |\Psi_j|^{p-1}\Psi_j-\epsilon_j\Psi_j\bigr)\ge0.
\]
The map $s\mapsto\rho_j(\mu_j |s|^{p-1}s-\epsilon_js)$ is nondecreasing
on $[-M-1,M+1]$ for all sufficiently large $j$.  Indeed, on this
interval the increasing map $s\mapsto|s|^{p-1}s$ has monotonicity constant
$p(M+1)^{p-1}$, and
$\mu_j p(M+1)^{p-1}-\epsilon_j>0$. Applying Green's formula to
the difference with test function $(w_j-\Psi_j)_+$ gives
$w_j\le \Psi_j$ in the half-disk.
The same comparison applies to $-w_j$, since the boundary
nonlinearity is odd.  Evaluation at $(x_0,0)$ therefore gives
$|w_j(x_0,0)|\le c\kappa_j$, uniformly in $x_0\in K$.
This proves \eqref{t134:int:eq:scaledbound}.  Pointwise statements here may
be understood using the continuous local representative: for each
fixed $j$ the bounded Robin datum on an interior flat segment
gives continuity of the harmonic solution up to that segment.
This barrier estimate depends only on the local supremum bound,
the weight bounds, and an upper bound for $\epsilon_j$; it does
not require convergence to $V$.

\medskip
\emph{(b) Harmonic Dirichlet replacement.}
It suffices, after translation and restriction, to prove the result
on $I_r$ with $r<R<2$.  Choose $R$ so that the entire diameter
$I_R$ lies in an interval on which \eqref{t134:int:eq:scaledbound} holds.
Let $D_j$ denote the bounded harmonic Dirichlet solution in
$B_R^+$ with boundary values $0$ on the open diameter and $w_j$
on the open semicircle.  The boundary data may jump at the two
corners.  This causes no difficulty: the bounded Dirichlet
solution is defined by harmonic measure; the two corner points
have harmonic measure zero.  We use $D_j$ only away from those
corners, where it is smooth up to the flat boundary and has odd
harmonic reflection.

The difference $w_j-D_j$ is the bounded harmonic extension of
$w_j$ on $I_R$ and $0$ on the semicircle.  Consequently
\begin{equation}\label{t134:int:eq:ajbounded}
 H_j:=\kappa_j^{-1}(w_j-D_j)
 \quad\hbox{is harmonic and satisfies}\quad
 \|H_j\|_{L^\infty(B_R^+)}\le C.
\end{equation}
Its trace on $I_R$ is
$z_j=\kappa_j^{-1}w_j(\cdot,0)$.  In particular,
$D_j-V\to0$ in $L^2(B_R^+)$, because
$w_j\to V$ there and $|w_j-D_j|\le C\kappa_j$.
Odd reflection followed by interior harmonic estimates yields
\begin{equation}\label{t134:int:eq:qconvergence}
 q_j:=\partial_yD_j(\cdot,0)
 \longrightarrow q:=\partial_yV(\cdot,0)
 \quad\hbox{in }C^k(I_{R'})
 \quad\hbox{for every }R'<R\hbox{ and every }k.
\end{equation}

Subtracting the equation for $D_j$ from
\eqref{t134:int:eq:original}, and using $\mu_j\kappa_j^p=1$, gives
\begin{equation}\label{t134:int:eq:scaledweak}
 \kappa_j\int_{B_R^+}\nabla H_j\cdot\nabla\varphi
 +\int_{I_R}\rho_j |z_j|^{p-1}z_j\varphi
 -\epsilon_j\kappa_j\int_{I_R}\rho_j z_j\varphi
 =\int_{I_R}q_j\varphi
\end{equation}
for all $H^1$ test functions supported away from the semicircle.
All uses of $H^1$ here are local away from the corners;
$H_j\in H^1$ on the support of any such test function because
$w_j\in H^1$ and $D_j$ is smooth there.

\medskip
\emph{(c) Strong convergence of the boundary trace.}
Let $\eta$ be a smooth cutoff supported away from the semicircle,
equal to one on the desired compact subinterval of $I_R$.
On a slightly larger flat interval put
\[
 g=\sgn(q)(|q|/\rho_0)^{1/p},
 \qquad\rho_0 |g|^{p-1}g=q.
\]
Since $1/p>1$ and $q$ is smooth, $g$ is $C^1$, including at
zeros of $q$.  Choose an $H^1\cap L^\infty$ extension $G$ whose
trace agrees with $g$ on the flat portion of
$\operatorname{supp}\eta$; its gradient may be chosen bounded.
Use $\varphi=\eta^2(H_j-G)$ in \eqref{t134:int:eq:scaledweak}.
After expanding the gradient term, this gives
\begin{align*}
 &\int_{I_R}\eta^2\bigl(\rho_j |z_j|^{p-1}z_j-q_j\bigr)(z_j-g)\\
 &\quad=-\kappa_j\int_{B_R^+}\eta^2|\nabla H_j|^2
       +\kappa_j\int_{B_R^+}\eta^2\nabla H_j\cdot\nabla G
       -2\kappa_j\int_{B_R^+}\eta(H_j-G)
                           \nabla H_j\cdot\nabla\eta\\
 &\qquad\quad+\epsilon_j\kappa_j
                  \int_{I_R}\rho_j\eta^2z_j(z_j-g)
 \le C\kappa_j.
\end{align*}
The last inequality follows by Young's inequality, absorption
into the negative gradient-square term, and the uniform bound
\eqref{t134:int:eq:ajbounded}; $\epsilon_j$, $\rho_j$, $z_j$, $G$, and
$\nabla G$ are bounded on the fixed supports.  Therefore
\begin{align}\label{t134:int:eq:minty}
 &\int_{I_R}\eta^2\rho_j
            \bigl(|z_j|^{p-1}z_j-|g|^{p-1}g\bigr)(z_j-g)\notag\\
 &\qquad\le C\kappa_j+
       \int_{I_R}\eta^2\bigl(q_j-\rho_j|g|^{p-1}g\bigr)(z_j-g)
 \longrightarrow0.
\end{align}
Here we used \eqref{t134:int:eq:qconvergence} and local uniform
convergence of $\rho_j$.

Choose $M_1$ to bound $|z_j|$ and $|g|$ on the support.
For $s,t\in[-M_1,M_1]$, absolute continuity of $s\mapsto|s|^{p-1}s$ gives
\[
 (|s|^{p-1}s-|t|^{p-1}t)(s-t)
 \ge pM_1^{p-1}|s-t|^2.
\]
Together with $\rho_j\ge\rho_*$, \eqref{t134:int:eq:minty} implies
$z_j\to g$ strongly in $L^2$ on the compact interval where
$\eta=1$.  Uniform boundedness then implies convergence in
every finite $L^\ell$.  This proves \eqref{t134:int:eq:strongtrace}.

\par\medskip
\Needspace{4\baselineskip}
\phantomsection\label{t134:int:step:contradiction}
\noindent\textbf{Step 8: The final sign contradiction.}

For the boundary traces $z_j$ in \hyperref[t134:int:step:boundary-trace]{Step~7},
\eqref{t134:int:eq:strongtrace} and $V=\sigma\operatorname{Im}z^N$
give, on each compact annulus of the real axis,
\[
z_j(x)\longrightarrow
\sigma\left(\frac N{\rho_0}\right)^{1/p}
\sgn(x)^{N-1}|x|^N
\quad\hbox{in every finite }L^\ell.
\]
Here $(N-1)/p=N$ was used; for $x<0$ the factor
$\sgn(x)^{N-1}$ equals $(-1)^{N-1}$.
For every fixed $T>0$, changing variables $x=\pm e^{-\tau}$ gives
\begin{align*}
&\int_0^T e^{-2N\tau}
\frac{v_+(s_j+\tau)-(-1)^Nv_-(s_j+\tau)}{d_j^{1/p}}
\dd\tau\\
&\hspace{12mm}=\int_0^T e^{-N\tau}
\bigl[z_j(e^{-\tau})-(-1)^N z_j(-e^{-\tau})\bigr]\dd\tau\\
&\hspace{12mm}\longrightarrow
2\sigma\left(\frac N{\rho_0}\right)^{1/p}
\int_0^T e^{-2N\tau}\dd\tau.
\end{align*}
The power $d_j^{1/p}$ follows from
$d_j\mu_j^{-1/p}=d_j^{1/p}$.

To pass to $T=\infty$, a uniform tail bound is required. At each future time
$s\ge s_j$, rescale $U$ at $e^{-s}$ and divide by $d_j$.
The preceding envelope gives uniform volume and Moser bounds on a
fixed larger half-disk. The penalty parameter remains
$\mu_j=d_j^{p-1}$ and the linear parameter is at most $\lambda r_j$.
The barrier estimate in \hyperref[t134:int:step:boundary-trace]{Step~7}(a) thus gives, uniformly for
all $s\ge s_j$,
\[
|v_+(s)|+|v_-(s)|\le C d_j^{1/p}.
\]
Consequently the tail of the integral above is bounded by
$C\int_T^\infty e^{-2N\tau}\dd\tau$, independently of $j$.

Applying \eqref{t134:int:eq:bint} now gives
\[
\frac{b'(s_j)}{d_j^{1/p}}
\longrightarrow
\frac{2N}{\pi}\,2\sigma
\left(\frac N{\rho_0}\right)^{1/p}\frac1{2N}
=\frac{2\sigma}{\pi}\left(\frac N{\rho_0}\right)^{1/p}.
\]
The limit has sign $\sigma$, contradicting
\eqref{t134:int:eq:bnegative}. Hence $\ord_0U\ne N$.

\end{proof}

\begin{proof}[\normalfont\bfseries Proof of Theorem~\ref{thm:integer-critical-order}]
Set $N=\kappa\in\N$. Since $0<p<1$, we have $N\ge2$ and
$p=1-N^{-1}$. By Lemma~\ref{t12:lem:dmo-chart}, there is a local
conformal map $\Phi$ with
\[
 \Phi(0)=q,\qquad
 0<c\le|\Phi'(z)|\le C\qquad(z\in\overline{B_R^+}).
\]
Set $U=u\circ\Phi$. Lemma~\ref{lem:conformal-equation} shows that
$U$ has the local $H^1\cap C$ regularity and solves the
boundary problem in Theorem~\ref{t134:int:thm:weighted}, with the continuous
positive weight $\rho=|\Phi'|$. Moreover, $U(0)=0$ and
$U\not\equiv0$ by the assumptions on $u$.

Since $\Phi$ is bi-Lipschitz and its Jacobian $|\Phi'|^2$ is bounded
above and away from zero, there are positive constants
$c_1,c_2,C_1,C_2$ such that
\[
 c_2\int_{\Om\cap B_{c_1r}(q)}u^2
 \le\int_{B_r^+}U^2
 \le C_2\int_{\Om\cap B_{C_1r}(q)}u^2
\]
for all sufficiently small $r>0$. Thus $\ord_0U=\ord_qu$.
Theorem~\ref{t134:int:thm:weighted} gives $\ord_0U\ne N$, and hence
$\ord_qu\ne\kappa$.
\end{proof}

\section{Supercritical vanishing order}\label{sec:supercritical}

\subsection{Integer order under \texorpdfstring{$C^1$}{C1} regularity}

\begin{lemma}\label{t15:c1:lem:cone}
Let $0<p<1$, $\lambda\ge0$, and choose $a\in(1/2,1)$.
There exists $\eta_a>0$ with the following property. Suppose that, in
$B_R(0)$, the domain is above a Lipschitz graph $y=\psi(x)$, with
$\psi(0)=0$ and $\operatorname{Lip}\psi\le\eta_a$. Let $u$ be a
continuous, locally $H^1$ harmonic function satisfying
\[
 \partial_\nu u-\lambda u+|u|^{p-1}u=0
\]
on the graph portion of the boundary, in the sense of normal traces.
Suppose $|u|\le M$ in
$\overline\Omega\cap B_R(0)$. If $\lambda R$ and
$M^{1-p}/R$ are sufficiently small, then
\begin{equation*}
 |u(0)|\le C_a
 M^{1/[1-a(1-p)]}R^{-a/[1-a(1-p)]}.
\end{equation*}
All smallness thresholds and constants can be chosen uniformly for
graphs with the indicated Lipschitz bound.
\end{lemma}

\begin{proof}
Choose $\eta_a>0$ so small that
\[
 a\bigl(\pi/2+\arctan\eta_a\bigr)<\pi/2.
\]
Set $D=1-a(1-p)>0$. For $0<h\le R$, use the principal branch to define
\[
 H_h(x,y)=\operatorname{Re}(h+y-ix)^a.
\]
Every point in $\overline\Omega\cap B_R$ satisfies $y\ge-\eta_a|x|$.
Consequently, uniformly in $h$,
\[
 |\arg(h+y-ix)|\le\pi/2+\arctan\eta_a,
 \qquad |h+y-ix|\ge c_a(h+|(x,y)|).
\]
The second estimate follows, for example, by homogeneity and
compactness on $h+|(x,y)|=1$, since $h+y-ix$ cannot vanish there.
Thus $H_h$ is harmonic on a neighborhood of the relevant closed
domain and
\begin{align*}
 c_a(h+|(x,y)|)^a&\le H_h(x,y)\le C_a(h+|(x,y)|)^a,\\
 |\nabla H_h(x,y)|&\le C_a h^{a-1}.
\end{align*}
The branch point $(0,-h)$ and the branch cut lie outside that domain.
Choose a fixed $C_0$ sufficiently large, and set
\[
 W=C_0MR^{-a}H_h.
\]
Then $W\ge M$ on the artificial circular boundary, $0<W\le C_1M$
throughout the closed domain, and on the graph
\[
 W\ge c_1M(h/R)^a,
 \qquad \partial_\nu W\ge-C_1MR^{-a}h^{a-1}.
\]
Take
\[
 h=R\left(K\frac{M^{1-p}}{R}\right)^{1/D},
\]
where $K$ is a sufficiently large fixed constant. The assumed
smallness makes $h\le R$. The identity
\[
 \frac{M^p(h/R)^{ap}}{MR^{-a}h^{a-1}}
 =M^{p-1}R^{a(1-p)}h^D=K
\]
shows, on increasing $K$, that $|W|^{p-1}W$ dominates twice the possible
negative normal derivative. Moreover,
$MR^{-a}h^{a-1}\ge M/R$, so the term $\lambda W$ is absorbed when
$\lambda R$ is sufficiently small. Hence
\[
 \partial_\nu W-\lambda W+|W|^{p-1}W\ge0.
\]

 If $w=(u-W)_+$, then $w$ has zero trace on the artificial
circle. Green's formula for $u-W$, the boundary inequality, and
monotonicity of $s\mapsto|s|^{p-1}s$ give
\[
 \int_{\Omega\cap B_R}|\nabla w|^2
 \le\lambda\int_{\partial\Omega\cap B_R}w^2.
\]
Vertical integration from each graph point to the upper circle gives
\[
 \int_{\partial\Omega\cap B_R}w^2
 \le 2R\sqrt{1+\eta_a^2}
       \int_{\Omega\cap B_R}|\nabla w|^2.
\]
Thus $u\le W$ for small $\lambda R$. Applying the same argument to
$-u$ gives $|u|\le W$. Continuity allows evaluation at the distinguished
graph point. Since $H_h(0,0)=h^a$,
\[
 |u(0)|\le C_0M(h/R)^a
 =C_aM^{1/D}R^{-a/D},
\]
as asserted.
\end{proof}

\begin{proposition}\label{t15:c1:prop:trace}
Suppose $\partial\Omega$ is $C^1$ near $q$, and let $u$ be a real
local solution satisfying
\eqref{paper:eq:local-problem}--\eqref{paper:eq:local-class}, with $u(q)=0$.
Assume that $k=\ord_q u$ is finite with $k>\kappa$.
Then for every
\begin{equation*}
 0<\beta<\frac{k-1}{p}
\end{equation*}
one has
\begin{equation*}
 |u(Q)|\le C_\beta |Q-q|^\beta
 \qquad(Q\in\partial\Omega,\ Q\to q).
\end{equation*}
\end{proposition}

\begin{proof}
By Proposition~\ref{t12:lem:physical}, every $s<k$ gives
$|u(X)|\le C_s|X-q|^s$ on the closed domain near $q$.
For $s>\kappa$ and $a<1$, put
\[
 B_a(s)=\frac{s-a}{1-a(1-p)}.
\]
Since
\[
 B_a(s)-s
 =\frac{a((1-p)s-1)}{1-a(1-p)}>0,
 \qquad \lim_{a\uparrow1,\ s\uparrow k}B_a(s)=\frac{k-1}{p},
\]
we may choose $s\in(\kappa,k)$ and $a\in(1/2,1)$ with
$B_a(s)>\beta$. Rotate the tangent at $q$ to be horizontal.
After restriction, the $C^1$ graph has Lipschitz constant at most
$\eta_a$. The same graph bound holds after translation to any nearby
boundary point $Q$. Apply Lemma~\ref{t15:c1:lem:cone} in $B_R(Q)$ with
$R=c|Q-q|$, for a fixed small $c>0$. On this ball,
$|u|\le C_sR^s=:M$. The condition
$M^{1-p}/R\to0$ follows from $s>\kappa$. The barrier gives
$|u(Q)|\le C R^{B_a(s)}$, which implies the required bound.
\end{proof}

\begin{lemma}\label{t15:lem:dirichlet}
Let $V$ be harmonic, locally $H^1$, and continuous on a closed
upper half-disk. Suppose that its trace satisfies
$|V(x,0)|\leq C|x|^\beta$, where $\beta>n$ and $n\geq0$ is an
integer. Then there are real coefficients $c_1,\ldots,c_n$ such that
\begin{equation*}
 V(z)=\sum_{j=1}^n c_j\operatorname{Im}z^j+O(|z|^\gamma)
 \qquad\bigl(\gamma<\min\{\beta,n+1\}\bigr)
\end{equation*}
uniformly up to the diameter. The sharper remainder is $O(r^\beta)$
if $\beta<n+1$, $O(r^{n+1}\log(1/r))$ if $\beta=n+1$, and
$O(r^{n+1})$ if $\beta>n+1$.
\end{lemma}

\begin{proof}
Write $h(x)=V(x,0)$ and choose a smooth cutoff $\chi$ equal to one
near zero. The Poisson extension
\[
 W(z)=\frac1\pi\operatorname{Im}
                 \int_\R\frac{\chi(t)h(t)}{t-z}\dd t
\]
belongs locally to $H^1$, because $\chi h\in H^{1/2}(\R)$.
The difference $V-W$ has zero trace near zero. Its odd reflection
is analytic and has a Taylor expansion in $\operatorname{Im}z^j$.

For $1\leq j\leq n$, the moments
\[
 a_j=\frac1\pi\int\chi(t)h(t)t^{-j-1}\dd t
\]
are absolutely convergent. Put $r=|z|$. The contribution of
$|t|\leq2r$ to the Poisson integral is $O(r^\beta)$, by positivity
and unit mass of the Poisson kernel, and the omitted portions of
these moments have the same bound after multiplication by $z^j$.
For $|t|>2r$, expand $(t-z)^{-1}$ through degree $n$. The
remainder is bounded by
\[
 C r^{n+1}\int_{2r}^{R_1}t^{\beta-n-2}\dd t,
\]
which gives the three asserted estimates. The zeroth complex
moment is real and contributes no imaginary part. Adding the
analytic odd-reflection remainder proves the statement, including
the closed-boundary uniformity.
\end{proof}

\begin{proof}[\normalfont\bfseries Proof of Theorem~\ref{t15:thm:main}\textup{(i)}]
Put $\rho(t)=|\Phi'(t)|$.
Proposition~\ref{t15:local-sucp} gives $m<\infty$.
Proposition~\ref{prop:physical} and Lemma~\ref{t12:lem:c1-chart}
give $\ord^{\infty}_0(U)=m$, $|\Phi(t)-q|\le C_\alpha|t|^\alpha$
for every $\alpha<1$, and $\rho\in L^Q$ locally for every finite
$Q>1$. The pullback boundary equation, in the sense of normal traces, is
$-U_y=\rho(\lambda U-|U|^{p-1}U)$.
Combining the estimate for the conformal map with Proposition~\ref{t15:c1:prop:trace}
gives
\begin{equation*}
 |U(t,0)|\le C_\beta|t|^\beta
 \quad\text{for every }\beta<(m-1)/p.
\end{equation*}
Indeed, choose $\beta<\beta'<(m-1)/p$ and $\alpha<1$ with
$\alpha\beta'>\beta$.
Take $m<\beta<(m-1)/p$ and $n=\lfloor m\rfloor$.
Lemma~\ref{t15:lem:dirichlet} gives
\[
 U(z)=\sum_{j=1}^n c_j\Ima z^j+O(|z|^\gamma),
 \qquad m<\gamma<\min\{\beta,n+1\}.
\]
The least nonzero degree equals the pointwise vanishing order, so
$m\in\N$ and $A=c_m\ne0$. Put $\ell=(m-1)/p>m$.
For $0<\delta<\min\{1,\ell-m\}$, applying the lemma again with
$n=m$ and $m+\delta<\beta<\ell$ proves
\eqref{t15:eq:conformal-expansion} and uniqueness of $A$.

Set $v_r=r^{-m}U(r\,\cdot)$, $P=A\Ima z^m$,
$w_r=v_r-P$, and $\tau_r=r^{1-m(1-p)}\to\infty$.
Since $P=0$ on $\R$, testing the difference equation with $\eta^2w_r$
gives
\begin{align*}
 \tfrac12\int\eta^2|\nabla w_r|^2
 +\tau_r\int\rho(rx)\eta^2|v_r|^{p+1}
 &\le C\int|w_r|^2|\nabla\eta|^2
 +\lambda r\int\rho(rx)\eta^2v_r^2\\
 &\quad+\left|\int(\partial_\nu P)\eta^2v_r\right|.
\end{align*}
The first and last terms vanish by uniform convergence on closed
half-disks. The linear term vanishes because $v_r$ is bounded and
$r\int_{-L}^L\rho(rx)\dd x=\int_{-Lr}^{Lr}\rho(t)\dd t\to0$.
Thus the convergence is strong in local $H^1$, and
\begin{equation}\label{t15:eq:weighted-penalty}
 \tau_r\int_{I_L}\rho(rx)|v_r(x,0)|^{p+1}\dd x\longrightarrow0.
\end{equation}
The physical trace bound follows from
Proposition~\ref{t15:c1:prop:trace} with $k=m$.

The Neumann trace in $H^{-1/2}$ is continuous under local $H^1$
convergence of harmonic functions, as follows by testing against
$H^1$ extensions of compactly supported $H^{1/2}$ boundary data.
Therefore $(v_r)_y\to mAx^{m-1}$ in local $H^{-1/2}$.
Moreover, for fixed finite $Q>1$ and compact $K$,
\[
 \|\lambda r\rho(r\,\cdot)v_r\|_{L^Q(K)}
 \le C\lambda r^{1-1/Q}\|\rho\|_{L^Q(rK)}\to0.
\]
The boundary equation gives
\begin{equation}\label{t15:eq:weak-flux}
 r^{1-m}\rho(r\,\cdot)|U(r\,\cdot,0)|^{p-1}U(r\,\cdot,0)
 \longrightarrow mA(\,\cdot\,)^{m-1}
 \quad\text{strongly in }H^{-1/2}_{\rm loc}(\R).
\end{equation}

The trace bound and the local graph length estimate imply
$\ell\le\ord_q^{\partial,\infty} u\le\ord_q^\partial u$.
If the last order exceeded $\ell$, some $\beta>\ell$ would satisfy
$\int_{\partial\Omega\cap B_R(q)}u^2\dd S\le CR^{2\beta+1}$.
For nonzero $0\le\varphi\in C_c^\infty((1,2))$,
\eqref{t15:eq:weak-flux} yields
\[
 r^{-m}\int_r^{2r}\rho(t)|U(t,0)|^{p-1}U(t,0)\varphi(t/r)\dd t
 \longrightarrow mA\int_1^2x^{m-1}\varphi(x)\dd x\ne0.
\]
Choose $\alpha<1$ with $\alpha(1+p\beta)>m$.
Then $\Gamma_r=\Phi([r,2r])\subset B_{C_\alpha r^\alpha}(q)$
and $|\Gamma_r|\le Cr^\alpha$. Since $\rho(t)\dd t=\dd S$,
H\"older's inequality gives the contradiction
\[
 cr^m\le C\int_{\Gamma_r}|u|^p\dd S
 \le C|\Gamma_r|^{1-p/2}
             \left(\int_{\Gamma_r}u^2\dd S\right)^{p/2}
 \le Cr^{\alpha(1+p\beta)}.
\]
Hence both boundary vanishing orders equal $\ell$.
\end{proof}

\subsection{The exact finer  asymptotic expansions}

\begin{lemma}\label{t15:w:comparison}
Let $\alpha,b\in L^\infty(I_R(x_0))$, $0\leq\alpha\leq\alpha^*$,
and $b\geq0$. For
\[
 \mathcal Bw=\partial_\nu w-\alpha(x)w+b(x)|w|^{p-1}w,
\]
a harmonic subsolution lies below a harmonic supersolution whenever
it does so on the artificial semicircle and $\alpha^*R<1$.
\end{lemma}

\begin{proof}
The positive part $w$ of the subsolution minus the supersolution
has zero trace on the semicircle. Applying Green's formula to
the difference, with $w$ as the test function, and using the
boundary inequalities gives
\[
 \int|\nabla w|^2\leq\alpha^*\int_{I_R(x_0)}w^2.
\]
For almost every vertical segment ending on the semicircle,
Cauchy--Schwarz gives
\[
 |w(x,0)|^2\leq\sqrt{R^2-(x-x_0)^2}
       \int_0^{\sqrt{R^2-(x-x_0)^2}}|w_y(x,y)|^2\dd y.
\]
After integration, the right-hand side of the preceding energy
inequality is at most $\alpha^*R\int|\nabla w|^2$. Thus $w=0$.
The nonlinear term is nonnegative solely by monotonicity of $s\mapsto|s|^{p-1}s$;
no upper bound on its effective derivative is needed.
\end{proof}

\begin{lemma}
\label{t15:w:fine}\label{t15d:lem:fine-comparison}
Let $\rho$ be continuous near zero with $0<\rho_*\le\rho\le\rho^*$,
and put $\rho_0=\rho(0)$. For an integer $m\ge1$ and $A\in\R$, set
\[
 P(z)=A\operatorname{Im}z^m,\qquad q_A(x)=mAx^{m-1},\qquad
 G_\rho(x)=\sgn(q_A(x))\left(\frac{|q_A(x)|}{\rho_0}\right)^{1/p}.
\]
Suppose that $v_r$ is harmonic on expanding upper half-disks,
locally $H^1$ and continuous up to their diameters, and
$v_r\to P$ locally uniformly on closed half-disks as $r\rightarrow0$.
Let $\eps_r>0$ tend to zero, with no prescribed relation to $r$,
and suppose that $v_r$ satisfies
\begin{equation}\label{t15:w:fine-law}
 \partial_\nu v_r-r\lambda\rho(rx)v_r
        +\rho(rx)\eps_r^{-p}|v_r|^{p-1}v_r=0.
\end{equation}
For each fixed $K>0$, choose a fixed $L>K+2$ and set
\[
 E_r=\|v_r-P\|_{C^0(\overline{B_L^+})},\qquad
 \zeta_r=\sup_{|x|\le L}|\rho(rx)-\rho_0|.
\]
Then, for all sufficiently small $r$,
\begin{align}
 &\left\|v_r(\cdot,0)/\eps_r-G_\rho
                                     \right\|_{C^0([-K,K])}
 \notag\\[-2pt]
 &\quad+\left\|\eps_r^{-p}|v_r(\cdot,0)|^{p-1}v_r(\cdot,0)-q_A/\rho_0
                                     \right\|_{C^0([-K,K])}
 \le C_K(E_r+\eps_r+\zeta_r).
                                  \label{t15d:eq:fine-linear}
\end{align}
The constant may depend on the fixed larger half-disk,
$p,m,A,\lambda,\rho_*$ and $\rho^*$, but is independent of
$r$, $\eps_r$, $E_r$ and $\zeta_r$. In particular,
\begin{equation}\label{t15:w:fine-trace}
 \eps_r^{-1}v_r(x,0)\longrightarrow G_\rho(x)
\end{equation}
locally uniformly on $\R$.
\end{lemma}

\begin{proof}
Fix $0<\alpha<\min\{1,1/p-1\}$. On bounded intervals
the map $s\mapsto\sgn(s)|s|^{1/p}$ is $C^{1,\alpha}$ with a
bounded norm. Thus, uniformly for $0\le\eta\le1$, the functions
\[
 g_\eta^\pm(x)=\sgn(q_A(x)\pm\eta)
      \left(\frac{|q_A(x)\pm\eta|}{\rho_0}\right)^{1/p}
\]
have bounded $C^{1,\alpha}$ norms on a larger fixed interval.
Multiply them by a fixed smooth cutoff equal to one on $[-L,L]$,
and let $Q_\eta^\pm$ be their upper half-plane Poisson extensions.
On the relevant closed half-disk and its diameter,
\begin{equation*}
 \|Q_\eta^\pm\|_{C^0(\overline{B_L^+})}
       +\|\partial_\nu Q_\eta^\pm\|_{C^0(I_L)}\le C_K.
\end{equation*}
For the derivative bound, the symmetric second difference
of a $C^{1,\alpha}$ boundary function is
$O(|t|^{1+\alpha})$. Its quotient by $t^2$ is integrable
at zero in the normal derivative formula for the Poisson
extension; the cutoff controls the tails. The bound is
uniform in $\eta$, including where $q_A\pm\eta=0$.

Choose one fixed radius $0<h\le1$ such that
$B_h^+(x_0)\subset B_L^+$ for all $|x_0|\le K$, and set
$M_r=C_K(E_r+\eps_r)$ with a sufficiently large constant.
The harmonic polynomial
\[
 H(X,Y)=X^2-Y^2+3Y
\]
satisfies $H\ge1$ on the unit upper semicircle,
$H(X,0)=X^2\ge0$ on the diameter, and
$\partial_\nu H=-3$ there. Define on $B_h^+(x_0)$
\[
 W_\pm(z)=P(z)+\eps_rQ_\eta^\pm(z)
       \pm M_rH((x-x_0)/h,y/h).
\]
The choice of $M_r$ brackets $v_r$ on the artificial semicircle.
On the diameter $P=0$, so monotonicity gives
\[
 \eps_r^{-p}|W_+|^{p-1}W_+\ge(q_A+\eta)/\rho_0,\qquad
 \eps_r^{-p}|W_-|^{p-1}W_-\le(q_A-\eta)/\rho_0.
\]
Consequently the boundary residual of $W_+$ is at least
\begin{align*}
 &-q_A+\eps_r\partial_\nu Q_\eta^+
      -3M_r/h-r\lambda\rho(rx)W_+
      +\frac{\rho(rx)}{\rho_0}(q_A+\eta)\\
 &\hspace{15mm}\ge c\eta-C_K(\eps_r+M_r/h+rM_r+\zeta_r).
\end{align*}
Here $|W_+(x,0)|\le C_KM_r$; the linear Robin term therefore
has size at most $C_KrM_r$. The lower residual has the opposite
sign with the same error. Since $h$ is fixed, choosing
$\eta=C_K(E_r+\eps_r+\zeta_r)$ makes $W_-$ a subsolution
and $W_+$ a supersolution for sufficiently small $r$; also
$\eta\le1$. Lemma~\ref{t15:w:comparison} applies because
$r\lambda\rho^*h<1$, and gives $W_-\le v_r\le W_+$.

At $z=x_0$ both $P$ and the added polynomial $H$ vanish.
The two exact corrector traces then imply
\[
 q_A(x_0)-\eta
 \le \rho_0 \eps_r^{-p}|v_r(x_0,0)|^{p-1}v_r(x_0,0)
 \le q_A(x_0)+\eta.
\]
This proves the nonlinear-flux estimate in
\eqref{t15d:eq:fine-linear} directly. The inverse map $s\mapsto\sgn(s)|s|^{1/p}$
is Lipschitz on the resulting bounded interval, so the trace
estimate follows with the same error. No application of the
$p$-H\"older map $s\mapsto|s|^{p-1}s$ to a trace error is made. Finally,
$E_r+\eps_r+\zeta_r\to0$ gives \eqref{t15:w:fine-trace}.

\end{proof}

\begin{proposition}[Exact trace with a continuous positive metric factor]
\label{t15:prop:weighted-trace}
Assume the conclusions of
Theorem~\ref{t15:thm:main}\textup{(i)} and suppose
additionally that $\rho$ is continuous near zero with
$0<\rho_*\le\rho\le\rho^*$. Put $\rho_0=\rho(0)$. Then
\begin{equation*}
 G_\rho(x)=\sgn(Ax^{m-1})\left(\frac{m|A|}{\rho_0}\right)^{1/p}|x|^\ell,
 \qquad r^{-\ell}U(rx,0)\longrightarrow G_\rho(x)
\end{equation*}
locally uniformly on $\R$. Furthermore,
\begin{align}
 r^{1-m}|U(rx,0)|^{p-1}U(rx,0)&\longrightarrow\frac{mA}{\rho_0}x^{m-1},
       \label{t15:eq:chart-flux}\\
 r^{1-m}\rho(rx)|U(rx,0)|^{p-1}U(rx,0)&\longrightarrow mAx^{m-1},
       \label{t15:eq:chart-weighted-flux}\\
 r^{1-m}U_y(rx,0)&\longrightarrow mAx^{m-1},
       \label{t15:eq:chart-normal}
\end{align}
locally uniformly. The last expression denotes the continuous
boundary extension of the interior derivative, as justified below.
\end{proposition}

\begin{proof}
Set $v_r(z)=r^{-m}U(rz)$ and $\eps_r=r^{\ell-m}\to0$.
The identities
\[
 p\ell=m-1,\qquad
 r^{1-m(1-p)}=\eps_r^{-p}
\]
rewrite the scaled Robin equation exactly as
\eqref{t15:w:fine-law}. On each fixed larger half-disk,
Theorem~\ref{t15:thm:main}\textup{(i)} gives $E_r\to0$,
and continuity of $\rho$ gives $\zeta_r\to0$. Hence
\eqref{t15d:eq:fine-linear} in Lemma~\ref{t15:w:fine} gives
\[
 \eps_r^{-1}v_r(x,0)=r^{-\ell}U(rx,0)\longrightarrow G_\rho(x)
\]
locally uniformly. Its nonlinear-flux estimate gives
\eqref{t15:eq:chart-flux} directly. Continuity of $\rho$
gives \eqref{t15:eq:chart-weighted-flux}. The exact
boundary equation reads
\[
 r^{1-m}U_y(rx,0)=\rho(rx)\left[
 r^{1-m}|U(rx,0)|^{p-1}U(rx,0)-\lambda r^{1+\ell-m}r^{-\ell}U(rx,0)\right].
\]
Since $1+\ell-m>0$, the last term tends uniformly to zero,
proving \eqref{t15:eq:chart-normal}. Lemma~\ref{t15:w:normal-trace}
below supplies the asserted classical meaning of $U_y$.
\end{proof}

\begin{lemma}\label{t15:w:normal-trace}
If $\rho$ is continuous and strictly positive and $U$ satisfies \eqref{t12:eq:pullback}, the interior function
$U_y$ extends continuously to compact subintervals of the diameter,
with value $\rho(|U|^{p-1}U-\lambda U)$.
\end{lemma}

\begin{proof}
On a diameter neighborhood write
$q(x)=\rho(x)(|U(x,0)|^{p-1}U(x,0)-\lambda U(x,0))$, which is continuous.
Choose a compactly supported cutoff equal to one on a smaller
interval and define
\[
 V(x,y)=\frac1{2\pi}\int\chi(t)q(t)
                          \log((x-t)^2+y^2)\dd t.
\]
This locally $H^1$ function is harmonic for $y>0$ and
\[
 V_y(x,y)=\frac1\pi\int
            \frac{y\chi(t)q(t)}{(x-t)^2+y^2}\dd t.
\]
The last expression is the Poisson extension of the continuous
datum $\chi q$, hence tends to $q$ uniformly on compact subintervals.
The difference $U-V$ has zero normal trace there and extends
harmonically by even reflection. Its $y$ derivative is continuous
and zero on the diameter. Therefore $U_y$ has exactly the asserted
continuous extension. No assertion concerning continuity of $U_x$
for an arbitrary continuous weight is needed.
\end{proof}

\begin{proof}[\normalfont\bfseries Proof of Theorem~\ref{t15:thm:main}\textup{(ii)}]
Under DMO, Lemma~\ref{t12:lem:dmo-chart} and
Proposition~\ref{t15:prop:weighted-trace} apply with
$\rho_0=a_0=|a_\Phi|$. Substituting
$\Phi^{-1}(q+w)=w/a_\Phi+o(|w|)$ into the conformal expansion gives
\[
 u(q+w)=B\Ima((\overline e w)^m)+o(|w|^m),
 \qquad B=A/a_0^m,
\]
uniformly on the closed  domain.
Put $U_r(z)=r^{-m}U(rz)$ and $P(z)=A\Ima z^m$.
For $Z_r=r^{-1}\Phi^{-1}\circ F_r$ and $Z_0(z)=z/a_0$,
$Z_r\to Z_0$ in $C^1$ on fixed closed half-disks, with uniformly
bounded Jacobians and inverse Jacobians. Splitting
\[
 U_r\circ Z_r-P\circ Z_0
 =(U_r-P)\circ Z_r+P\circ Z_r-P\circ Z_0
\]
transfers the uniform and strong $H^1$ limits. Convergence of
the scaled domains and polar integration give the density
$\int_{B_1^+}(B\Ima z^m)^2=\pi B^2/[4(m+1)]$.

On the boundary, $t(s)=\Phi^{-1}(\gamma(s))=s/a_0+o(s)$.
The uniform conformal trace limit therefore gives
\[
 r^{-\ell}u(\gamma(rx))
 \longrightarrow\sgn(Ax^{m-1})
      \left(\frac{m|A|}{a_0^m}\right)^{1/p}|x|^\ell
 =G_B(x).
\]
Since $p\ell=m-1$, this implies the nonlinear flux limit.
Lemma~\ref{t15:lem:boundary-c1} makes the Robin law pointwise;
the inward normal derivative is $|u|^{p-1}u-\lambda u$ and its linear
term, after scaling, is $O(r^{1+\ell-m})$. This proves the
normal-derivative limit and completes part \textup{(ii)}.
\end{proof}

\section*{Acknowledgments}
Liu's research is supported by the National Natural Science Foundation of China
(No.~12471198). Yang's research is supported by the National Natural Science
Foundation of China (No.~12431018).
\vspace{5mm}

{\bf{Declaration of AI use.}} During the preparation of this manuscript, the authors used ChatGPT
(OpenAI) as an auxiliary tool for language editing and preliminary examination of mathematical
derivations. All mathematical statements, proofs, references, and conclusions were independently
reviewed and verified by the authors, who take full responsibility for the accuracy and integrity of the manuscript.


\begin{thebibliography}{99}

\bibitem{AE}
V. Adolfsson and L. Escauriaza,
\emph{$C^{1,\alpha}$ domains and unique continuation at the boundary},
Comm. Pure Appl. Math. \textbf{50} (1997), no.~10, 935--969.


\bibitem{AEK}
V. Adolfsson, L. Escauriaza and C.~E. Kenig,
\emph{Convex domains and unique continuation at the boundary},
Rev. Mat. Iberoam. \textbf{11} (1995), no.~3, 513--525.

\bibitem{BanerjeeGarofalo}
A. Banerjee and N. Garofalo,
\emph{Quantitative uniqueness for elliptic equations at the boundary
of $C^{1,\mathrm{Dini}}$ domains},
J. Differential Equations \textbf{261} (2016), no.~12, 6718--6757.

\bibitem{BR}
M.~S. Baouendi and L.~P. Rothschild,
\emph{Harmonic functions satisfying weighted sign conditions on the boundary},
Ann. Inst. Fourier (Grenoble) \textbf{43} (1993), no.~5, 1311--1318.

\bibitem{CaffarelliSilvestre}
L. Caffarelli and L. Silvestre,
\emph{An extension problem related to the fractional Laplacian},
Comm. Partial Differential Equations \textbf{32} (2007),
no.~7--9, 1245--1260.




\bibitem{DongLeeKim}
H. Dong, J. Lee and S. Kim,
\emph{On conormal and oblique derivative problem for elliptic equations
with Dini mean oscillation coefficients},
Indiana Univ. Math. J. \textbf{69} (2020), no.~6, 1815--1853.


\bibitem{GarofaloLin}
N. Garofalo and F.-H. Lin,
\emph{Unique continuation for elliptic operators:
a geometric-variational approach},
Comm. Pure Appl. Math. \textbf{40} (1987), no.~3, 347--366.

\bibitem{GP}
A. Girouard and I. Polterovich,
\emph{Spectral geometry of the Steklov problem (Survey article)},
J. Spectr. Theory \textbf{7} (2017), no.~2, 321--359.

\bibitem{JerisonKenig}
D. Jerison and C.~E. Kenig,
\emph{Unique continuation and absence of positive eigenvalues
for Schr\"odinger operators},
Ann. of Math. (2) \textbf{121} (1985), no.~3, 463--488.

\bibitem{Kalaj}
D. Kalaj,
\emph{Muckenhoupt weights and Lindel\"of theorem for harmonic mappings},
Adv. Math. \textbf{280} (2015), 301--321.


\bibitem{KenigZhaoSingular}
C. Kenig and Z. Zhao,
\emph{Boundary unique continuation on $C^1$-Dini domains and
the size of the singular set},
Arch. Ration. Mech. Anal. \textbf{245} (2022), no.~1, 1--88.

\bibitem{KenigZhaoExpansion}
C. Kenig and Z. Zhao,
\emph{Expansion of harmonic functions near the boundary of Dini domains},
Rev. Mat. Iberoam. \textbf{38} (2022), no.~7, 2117--2152.

\bibitem{KriventsovLi}
D. Kriventsov and Z. Li,
\emph{Asymptotic expansions for harmonic functions at conical boundary points},
Rev. Mat. Iberoam. \textbf{41} (2025), no.~3, 837--866.

\bibitem{KukavicaNystrom}
I. Kukavica and K. Nystr\"om,
\emph{Unique continuation on the boundary for Dini domains},
Proc. Amer. Math. Soc. \textbf{126} (1998), no.~2, 441--446.

\bibitem{Li}
Z. Li,
\emph{Unique continuation for Robin problems with non-smooth potentials},
J. Funct. Anal. \textbf{288} (2025), no.~6, Paper No.~110811.

\bibitem{LiWang}
Z. Li and W. Wang,
\emph{Quantitative unique continuation for Robin boundary value
problems on $C^{1,1}$ domains},
Indiana Univ. Math. J. \textbf{72} (2023), no.~4, 1429--1460.

\bibitem{LinNodal}
F.-H. Lin,
\emph{Nodal sets of solutions of elliptic and parabolic equations},
Comm. Pure Appl. Math. \textbf{44} (1991), no.~3, 287--308.

\bibitem{Mori}
A. Mori,
\emph{On an absolute constant in the theory of quasi-conformal mappings},
J. Math. Soc. Japan \textbf{8} (1956), no.~2, 156--166.

\bibitem{PariniWeth}
E. Parini and T. Weth,
\emph{Existence, unique continuation and symmetry of least energy
nodal solutions to sublinear Neumann problems},
Math. Z. \textbf{280} (2015), 707--732.

\bibitem{Pom}
C. Pommerenke,
\emph{Boundary Behaviour of Conformal Maps},
Grundlehren der mathematischen Wissenschaften, vol.~299,
Springer-Verlag, Berlin, 1992.

\bibitem{Ruland}
A. R\"uland,
\emph{Unique continuation for sublinear elliptic equations
based on Carleman estimates},
J. Differential Equations \textbf{265} (2018), no.~11, 6009--6035.

\bibitem{SoaveTerracini}
N. Soave and S. Terracini,
\emph{The nodal set of solutions to some elliptic problems:
sublinear equations, and unstable two-phase membrane problem},
Adv. Math. \textbf{334} (2018), 243--299.

\bibitem{SoaveTortone}
N. Soave and G. Tortone,
\emph{On the nodal set of solutions to some sublinear equations
without homogeneity},
Arch. Ration. Mech. Anal. \textbf{248} (2024), Paper No.~26.

\bibitem{SoaveWeth}
N. Soave and T. Weth,
\emph{The unique continuation property of sublinear equations},
SIAM J. Math. Anal. \textbf{50} (2018), no.~4, 3919--3938.

\bibitem{Tolsa}
X. Tolsa,
\emph{Unique continuation at the boundary for harmonic functions
in $C^1$ domains and Lipschitz domains with small constant},
Comm. Pure Appl. Math. \textbf{76} (2023), no.~2, 305--336.

\bibitem{Tortone}
G. Tortone,
\emph{The nodal set of solutions to some nonlocal sublinear problems},
Calc. Var. Partial Differential Equations \textbf{61} (2022),
Paper No.~82.

\bibitem{Vita}
S. Vita,
\emph{Boundary unique continuation in planar domains by conformal mapping},
Bull. Lond. Math. Soc. \textbf{58} (2026), no.~5, e70356.

\bibitem{Warschawski}
S.~E. Warschawski,
\emph{On differentiability at the boundary in conformal mapping},
Proc. Amer. Math. Soc. \textbf{12} (1961), 614--620.

\bibitem{Weiss}
G.~S. Weiss,
\emph{A homogeneity improvement approach to the obstacle problem},
Invent. Math. \textbf{138} (1999), no.~1, 23--50.

\bibitem{ZhuZhuge}
J. Zhu and J. Zhuge,
\emph{Nodal sets of Dirichlet eigenfunctions in quasiconvex Lipschitz domains},
Adv. Math. \textbf{502} (2026), Paper No.~111150.

\end{thebibliography}
\end{document}